\documentclass{amsart}

\title{Homeomorphic subsurfaces and the omnipresent arcs}
\author{Federica Fanoni}
\address[Federica Fanoni]{CNRS, Univ Paris Est Creteil, Univ Gustave Eiffel, LAMA, F-94010 Creteil, France}
\email{federica.fanoni@gmail.com}
\author{Tyrone Ghaswala}
\address[Tyrone Ghaswala]{D\'epartment de math\'ematiques, Universit\'e du Qu\'ebec \`a Montr\'eal, Montr\'eal, Canada}
\email{ty.ghaswala@gmail.com}
\author{Alan \mc Leay}
\address[Alan \mc Leay]{Mathematics Research Unit, Universit\'e du Luxembourg, 4365 Esch-sur-Alzette, Luxembourg}
\email{mcleay.math@gmail.com}
\date{\today}

\usepackage{amsmath} % Necessary
\usepackage{amssymb} % Necessary
\usepackage{amsthm} % Necessary
\usepackage{amsfonts} % Necessary
\usepackage[usenames,dvipsnames,svgnames]{xcolor} % For tonnes of colours and control over creating colours
\usepackage{hyperref} % Automatically turns all internal references to links
\hypersetup{colorlinks=true,allcolors=blue!75!black} % Changes links to be dark blue with no annoying boxes around them
\usepackage{mathrsfs} % Allows the use of \mathscr letters
\usepackage{enumitem}
\usepackage{overpic} % For writing over figures
\usepackage[margin=1.2in]{geometry}
\usepackage[normalem]{ulem}

\newcommand{\NN}{\mathbb{N}}
\newcommand{\RR}{\mathbb{R}}
\newcommand{\ZZ}{\mathbb{Z}}

\newcommand{\vA}{\mathcal{A}}
\newcommand{\vF}{\mathcal{F}}

\newcommand{\vU}{\mathcal{U}}
\newcommand{\vX}{\mathcal{X}}

\newcommand{\Ends}{\operatorname{Ends}}
\newcommand{\Endsg}{\Ends_g}
\newcommand{\Endsp}{\Ends_p}
\newcommand{\genus}{\operatorname{genus}}
\newcommand{\Homeo}{\operatorname{Homeo}}
\newcommand{\MCG}{\operatorname{MCG}}
\newcommand{\id}{\text{id}}

\newcommand{\sm}{\setminus}
\newcommand{\st}{\;|\;}

\DeclareMathOperator{\sqleq}{\preccurlyeq} %squiggly-leq
\DeclareMathOperator{\sqgeq}{\succcurlyeq} %squiggly-geq

\newcommand{\abs}[1]{\left\lvert #1 \right\rvert}

\definecolor{lightgrey}{gray}{.85}

\theoremstyle{plain}
\newtheorem{thm}{Theorem}[section]
\newtheorem{lemma}[thm]{Lemma}
\newtheorem{cor}[thm]{Corollary}
\newtheorem{prop}[thm]{Proposition}
\newtheorem*{quest}{Question}

\theoremstyle{definition}
\newtheorem{defn}[thm]{Definition}
\newtheorem{rmk}[thm]{Remark}
\newtheorem{eg}[thm]{Example}
\theoremstyle{plain}
\newtheorem{thmintro}{Theorem}

\makeatletter
\newcommand{\mc}{}
\DeclareRobustCommand{\mc}{%
  M%
  \raisebox{\dimexpr\fontcharht\font`M-\height}{%
    \check@mathfonts\fontsize{\sf@size}{0}\selectfont
    c%
  }%
}

\makeatletter
\def\paragraph{\@startsection{paragraph}{4}%
  \z@\z@{-\fontdimen2\font}%
  {\normalfont\bfseries}}
\makeatother

\begin{document}

\begin{abstract}
In this article, we are concerned with various aspects of arcs on surfaces. In the first part, we deal with topological aspects of arcs and their complements. We use this understanding, in the second part, to construct an interesting action of the mapping class group on a subgraph of the arc graph. This subgraph naturally emerges from a new characterisation of infinite-type surfaces in terms of homeomorphic subsurfaces.
\end{abstract}

\maketitle

\section{Introduction}
Our work stems from the following simple observation: no surface $\Sigma$ of finite-type admits a subsurface $S$ homeomorphic to $\Sigma$, but not homotopic to it\footnote{I.e.\ such that the inclusion map is not homotopic to a homeomorphism.}. This holds for a quite general type of subsurface -- the boundary of $S$ in $\Sigma$ can be any union of (pairwise disjoint and essential) simple arcs and curves on $\Sigma$. However, the proof of this fact relies on the finiteness of genus and number of punctures of a finite-type surface, so it does not carry over to the infinite-type setting (i.e. where the fundamental group is not finitely generated). As such, a natural question is: given an infinite-type surface $\Sigma$, does there exist a subsurface homeomorphic to $\Sigma$ such that the inclusion map is not homotopic to a homeomorphism? Our first result states that not only are infinite-type surfaces characterised by having such a subsurface, but also that we can restrict our attention to a simple class of subsurfaces.

\begin{thmintro}\label{thmintro:one-cut-homeo}
A surface $\Sigma$ is of infinite-type if and only if there exists a separating essential proper simple arc $\alpha$ such that one component of $\Sigma\sm\alpha$ is homeomorphic to $\Sigma$.
\end{thmintro}

We call a subsurface as in the theorem a \emph{one-cut homeomorphic subsurface}. The existence of such subsurfaces provides us with a way to select a special class of arcs, which we want to think of as ``truly essential''.  Indeed, for a finite-type surface with at least one end, the essential arcs are those intersecting every homeomorphic subsurface (although in this case the inclusion of the subsurface in the full surface is homotopic to a homeomorphism). Following Theorem \ref{thmintro:one-cut-homeo}, we extend this viewpoint to the infinite-type setting; we say that an arc joining distinct ends is \emph{omnipresent} if it intersects \emph{every} one-cut homeomorphic subsurface.

One might hope to be able to characterise omnipresent arcs on any surface, but, as discussed in Section \ref{sec:stability}, this seems to be impossible in general. On the other hand, we are able to give such a characterisation for a subclass of surfaces, which we call \emph{stable}. The ends of these surfaces are required to satisfy a \emph{stability} condition (see Section \ref{sec:stability}), which turns out to be the same as the one considered by Mann and Rafi in their recent work on the coarse geometry of mapping class groups \cite{MR_Large}. The set of stable surfaces is a large subset of the set of all surfaces -- in particular, it has the cardinality of the continuum (Remark \ref{rmk:alot-of-stables}).

As the next theorem shows, ends with finite mapping class group orbit play an important role in the study of stable surfaces; we call such ends \emph{finite-orbit ends}.

\begin{thmintro}\label{thmintro:omnipresent&fends}
Let $\Sigma$ be a stable surface. Then an arc joining two distinct ends is omnipresent if and only if both its ends are finite-orbit ends.
\end{thmintro}

Recall that the \emph{arc graph} $\vA(\Sigma)$ has vertices corresponding to isotopy classes of all essential arcs, and edges between vertices with disjoint representatives.  In analogy with the finite-type case, we consider the \emph{omnipresent arc graph} $\Omega(\Sigma)$, the full subgraph of $\vA(\Sigma)$ spanned by all omnipresent arcs, that is, the ``truly essential'' arcs.  The reason to look at subgraphs of the arc graph is the following. In the finite-type case, curve and arc graphs have been extremely useful to understand properties of mapping class groups -- think for instance of quasi-isometric rigidity \cite{BKMM_Geometry, Bowditch_Large-scale}, or cohomological properties \cite{Harer_Stability, Harer_Virtual, BF_Bounded}. While we can define these graphs for infinite-type surfaces as well, one issue for many applications is that they have finite diameter. As such, there has been interest in constructing other graphs associated to infinite-type surfaces with a good mapping class group action to circumvent this problem \cite{FP_Curve, Bavard_Hyperbolicite, AFP_Arc, DFV_Big} and the definition of the omnipresent arc graph fits in this framework. We show:

\begin{thmintro}\label{thmintro:omnipresent-hyperbolic}
For any stable surface $\Sigma$ with at least three finite-orbit ends, $\Omega(\Sigma)$ is a connected $\delta$-hyperbolic graph on which $\MCG(\Sigma)$ acts with unbounded orbits. The constant $\delta$ can be chosen independently of $\Sigma$.
\end{thmintro}

An interesting observation is that, in general, the action of the mapping class group is not continuous (Proposition \ref{prop:discontinuous-action}) and the set of omnipresent arcs is uncountable (Lemma \ref{lem:uncountable}).

We remark that the omnipresent arc graph is naturally associated to the surface, in the sense that there are no (surface-dependent) choices to be made. This is in contrast with the construction in \cite{DFV_Big}, where the graphs depend on the choice of a collection of closed subsets of the space of ends of the surface. While we can define the omnipresent arc graph for every surface, it is in the setting of (most) stable surfaces that we can certify that the graph has the good properties stated in the theorems, thus giving a partial answer to the following question:

\begin{quest}\cite[Problem 2.1]{aim}
What combinatorial objects are ``good'' analogues of the curve complex, either uniformly for all infinite-type surfaces or for some class of infinite-type surfaces? Here ``good'' means that there exist relationships between topological properties of the mapping class and dynamical properties of its action on the combinatorial object.
\end{quest}

\subsection{Other definitions of graphs}
Our strategy to prove Theorem \ref{thmintro:omnipresent-hyperbolic} views omnipresent arc graphs as a special case of a more general construction, reminiscent of that of \cite{DFV_Big}. More precisely, given a surface $\Sigma$ and a subset $P$ of ends, we define $\vA_2(\Sigma,P)$ to be the subgraph of the arc graph spanned by arcs with two distinct ends in $P$. We show that for stable surfaces, the omnipresent arc graph corresponds to $\vA_2(\Sigma,\vF)$, where $\vF$ is the collection of finite-orbit ends.  We then show that a theorem analogous to Theorem \ref{thmintro:omnipresent-hyperbolic} holds for this graph if $P$ is finite, of cardinality at least three, and mapping class group invariant (Theorem \ref{thm:2-ended-arc-graph}).

In \cite{DFV_Big}, the authors introduce a topological invariant associated to a surface, the finite-invariance index $\mathfrak{f}(\Sigma)$, and show that if $\mathfrak{f}(\Sigma)\geq 4$, then there is a connected graph on which $\MCG(\Sigma)$ acts (continuously) with unbounded orbits. It is easy to check that the surfaces in Theorem \ref{thmintro:omnipresent-hyperbolic} satisfy $\mathfrak{f}(\Sigma)\geq |P|$. As we require $|P|$ to be only at least three, our result includes surfaces not covered by \cite[Theorem 2]{DFV_Big}. One such example is the surface with three ends, all nonplanar, for which it was shown in \cite{DFV_Big} that there is no graph, whose vertices are homotopy classes of curves, which is connected and on which the mapping class group acts with unbounded orbits. On the other hand, there are surfaces with $\mathfrak{f}(\Sigma)\geq 4$ which are not covered by our theorem, though for these it is not known whether the graph constructed in \cite{DFV_Big} is Gromov hyperbolic.

Generalising work of Bavard \cite{Bavard_Hyperbolicite}, Aramayona, Fossas and Parlier proved in \cite{AFP_Arc} a result analogous to Theorem \ref{thmintro:omnipresent-hyperbolic} for the subgraph of the arc graph given by arcs with both endpoints in a given finite set $P$ of isolated planar ends. While they don't need a lower bound on the size of $P$, it is important in their construction that the ends in $P$ are planar and isolated. As in our case the ends are not necessarily planar nor isolated, the behavior of a finite collection of arcs is not captured by their behavior in a finite-type surface. This in turn implies that the strategy to show hyperbolicity in \cite{AFP_Arc} and \cite{DFV_Big} -- reducing to the finite-type setting -- won't carry over to our setup. We need instead to understand how to adapt the unicorn construction of \cite{HPW_Slim} to pairs of arcs that intersect infinitely many times (Section \ref{sec:arcgraphs}).

\subsection{Plan of the paper}

In Section \ref{sec:arcs&complements} we discuss topological properties of arcs and their complements, proving Theorem \ref{thmintro:one-cut-homeo}. In Section \ref{sec:omnipresent} we define omnipresent arcs and show that if an arc joins two finite-orbit ends, then it is an omnipresent arc (one direction of Theorem \ref{thmintro:omnipresent&fends}). We then discuss in Section \ref{sec:stability} the stability conditions necessary to be able to characterise omnipresent arcs and prove the other direction of Theorem \ref{thmintro:omnipresent&fends}. In Section \ref{sec:arcgraphs} we define the arc graphs we are interested in, discuss the (dis)continuity of the mapping class group action and prove Theorem \ref{thmintro:omnipresent-hyperbolic}.
\subsection*{Acknowledgements}
This work started during the AIM workshop \emph{Surfaces of infinite-type}. We are grateful to the American Institute of Mathematics and the National Science Foundation for the hospitality and support. We are also thankful to Kathryn Mann and Kasra Rafi for sharing their manuscript of \cite{MR_Large}. We thank the referees for their useful comments and for pointing out a mistake in a previous version of the article.

The first author would like to thank Sebastian Hensel for useful discussions.  The second author would like to thank Ian Payne for useful discussions, and is grateful for support from a PIMS postdoctoral fellowship at the University of Manitoba and a CIRGET postdoctoral fellowship at the Universit\'e du Qu\'ebec \`a Montr\'eal. The third author was supported by the COALAS FNR AFR bilateral grant.

\section{Preliminaries}\label{sec:preliminaries}
We recall here some basic definitions, to establish notation, and we introduce objects and conventions we will use in the sequel.

\subsection{Surfaces and ends}
By \emph{surface} we mean a two-dimensional, connected, orientable manifold. Unless otherwise stated, surfaces will have no boundary, except in the case of properly embedded subsurfaces. A surface is of \emph{finite type} if its fundamental group is finitely generated and of \emph{infinite type} otherwise. A simple closed curve on a surface is \emph{essential} if it does not bound a disk or a once-punctured disk.

The \emph{mapping class group} $\MCG(\Sigma)$ of a surface $\Sigma$ is the group of orientation preserving homeomorphisms, up to isotopy. It is a topological group with respect to the topology generated by sets of the form $\phi\cdot U_A$, where $\phi$ is a mapping class, $A=\{\gamma_1,\dots,\gamma_n\}$ a finite collection of homotopy classes of essential simple closed curves and
$$U_A=\{\psi\in\MCG(\Sigma)\st \psi(\gamma_i)=\gamma_i \;\forall i=1,\dots,n\}.$$
This is the same as the quotient topology induced by the compact-open topology on the group of (orientation-preserving) homeomorphisms (a fact that can be shown using the Alexander method).

Given a surface $\Sigma$, we define its \emph{ends} to be equivalence classes of admissible descending chains. An \emph{admissible descending chain} is a nested sequence $U_1 \supset U_2 \supset \cdots$, where
\begin{enumerate}
\item each $U_n$ is either a connected open unbounded\footnote{A set is unbounded if its closure is not compact.}  set with compact boundary or the closure of such a set, and
\item for every compact set $K$ of $\Sigma$, $U_n \cap K = \emptyset$ for any $n$ large enough.
\end{enumerate}
Two such chains  $U_1 \supset U_2 \supset \cdots$ and $V_1 \supset V_2\supset \cdots$ are equivalent if for any $n$ there exists an $N$ such that $U_N \subset V_n$, and for any $m$ there exists an $M$ such that $V_M \subset U_m$.  We will often say that an admissible descending chain \emph{defines} the end $[U_1\supset U_2\supset\dots]$.

We denote by $\Ends(\Sigma)$ the set of ends of the surface, endowed with the topology with basis
$$\{U^* \st \mbox{int}(U) \mbox{  is an open set with compact boundary}\},$$
where
$$U^*:=\{ [ V_1 \supset V_2 \supset \cdots] \st \exists n : V_n \subset U\}.$$

An end $[ V_1 \supset V_2 \supset \cdots]$ is {\emph{planar}} if $V_n$ is planar (has genus zero) for large enough $n$. Otherwise, the end is \emph{nonplanar}. An end is a \emph{puncture} if it is planar and isolated. We denote by $\Endsp(\Sigma)$ and $\Endsg(\Sigma)$ the subspaces of $\Ends(\Sigma)$ given by planar and nonplanar ends respectively.

Adding $\Ends(\Sigma)$ to the surface yields a compactification of $\Sigma$, called \emph{Freudenthal compactification}. Moreover, the actions of $\Homeo(\Sigma)$ and $\MCG(\Sigma)$ extend to an action on the space of ends and on the Freudenthal compactification of $\Sigma$. We will regularly abuse notation and simply write $\phi(e)$ for the image of an end $e$ by a mapping class $\phi$ via the action just described. We remark that by neighbourhood of an end we will mean a neighbourhood in $\Ends(\Sigma)$ (and not in the Freudenthal compactification of $\Sigma$).

Surfaces are characterised, up to homeomorphism, by their genus and the pair of topological spaces $(\Ends(\Sigma),\Endsg(\Sigma))$ (see \cite{Richards_Classification}). Note that Richards' proof of the main theorem in \cite{Richards_Classification} also shows that given two surfaces $\Sigma$ and $\Sigma'$ with the same genus and a homeomorphism $$f:(\Ends(\Sigma),\Endsg(\Sigma))\to (\Ends(\Sigma'),\Endsg(\Sigma'))$$ of pairs of topological spaces, there is a homeomorphism $F:\Sigma\to\Sigma'$ inducing $f$ on the spaces of ends. In particular, this implies that the mapping class group orbit of an end of a surface $\Sigma$ coincides with its orbit under the action of $\Homeo(\Ends(\Sigma),\Ends_g(\Sigma))$.

Furthermore, for any surface $\Sigma$, $\Ends(\Sigma)$ is (homeomorphic to) a closed subset of the Cantor set and $\Endsg(\Sigma)$ is closed in $\Ends(\Sigma)$. Conversely, for any pair $(E,F)$ of topological spaces, where $E$ is a closed subset of the Cantor set and $F$ a closed subset of $E$, there is a surface $\Sigma$ with $(\Ends(\Sigma),\Endsg(\Sigma))\simeq (E,F)$. If $F\neq \emptyset$, the surface is unique (up to homeomorphism) and has infinite genus. Otherwise, there is one such surface for every finite genus. In particular, the set of homeomorphism classes of surfaces has the cardinality of the continuum.

\subsection{Arcs and subsurfaces}

By \emph{arc} we mean the image of a proper embedding $a:I\to\Sigma$, where $I$ is either $[0,1]$, $[0,1)$ or $(0,1)$, such that:
\begin{itemize}
\item if $x$ is an endpoint of $I$, $a(x)$ belongs to a boundary component of $\Sigma$,
\item the closure of $a(I)$ in the Freudenthal compactification of $\Sigma$ is not contractible. 
\end{itemize}
We will abuse notation and conflate an arc with its homotopy class (where homotopies are required to fix the boundary pointwise). Given the (homotopy class of an) arc $\alpha$, we define $\partial\alpha$ to be the set of ends and/or boundary components joined by $\alpha$. Given an arc $\alpha$ with $\partial\alpha\subset\Ends(\Sigma)$, if $\partial\alpha$ is a single end $e$ we say that $\alpha$ is a \emph{loop} (\emph{based at $e$}), while if $|\partial \alpha|=2$ we say that $\alpha$ is \emph{2-ended}.

Given an arc $\alpha$ and a finite union of simple closed curves $\beta$, the \emph{geometric intersection number} of $\alpha$ and $\beta$, denoted $i(\alpha,\beta)$, is the minimum number of intersections between representatives of the two homotopy classes.

We will be concerned with two types of subsurfaces: properly embedded and non-properly embedded ones. Both types will often be considered up to homotopy.

The non-properly embedded subsurfaces we will consider are complements of arcs and curves. Note that in the literature a subsurface obtained by cutting along an arc or a curve is usually a surface with boundary and hence it is properly embedded. Here instead, given an arc or a curve, we will simply consider the (open) subsurface(s) of $\Sigma$ given by its complementary component(s).

We require properly embedded subsurfaces to have homotopically nontrivial boundary components (but not necessarily essential).
Flare surfaces and subsurfaces in an exhaustion are the main properly embedded subsurfaces that will play a role in our work. A \emph{flare surface} is a properly embedded unbounded subsurface of $\Sigma$ whose boundary is a single separating simple closed curve. Note that this is a mild generalisation of the notion of flare surface used in \cite{FHV_Big}, where the closure of the complementary component of a flare surface is not allowed to be a finite-type surface with at most one planar end. With the definition we consider, we have the following result, which we will use throughout without explicit mention:

\begin{lemma}
Every proper clopen subset of $\Ends(\Sigma)$ can be realised as $X^*$ for some flare surface $X$.
\end{lemma}
\begin{proof}
Let $U$ be a proper clopen subset of $\Ends(\Sigma)$. Then there exist two surfaces $\Sigma_1$ and $\Sigma_2$ with the following properties:
\begin{itemize}
    \item each has one compact boundary component;
    \item there are homeomorphisms $$\varphi_1:(\Ends(\Sigma_1),\Endsg(\Sigma_1))\to (U,U\cap\Endsg(\Sigma))$$
and
$$\varphi_2:(\Ends(\Sigma_2),\Endsg(\Sigma_2))\to (\Ends(\Sigma)\sm U,\Endsg(\Sigma)\sm U);$$
\item $\genus(\Sigma_1)+\genus(\Sigma_2)=\genus(\Sigma)$.
\end{itemize}

Let $\Sigma'$ be the surface obtained by gluing $\Sigma_1$ and $\Sigma_2$ along their boundary components. Then $\genus(\Sigma')=\genus(\Sigma)$ and the map $$\varphi:(\Ends(\Sigma'),\Endsg(\Sigma'))\to (\Ends(\Sigma),\Endsg(\Sigma))$$
induced by $\varphi_1$ and $\varphi_2$ is a homeomorphism. So, by \cite{Richards_Classification}, there is a homeomorphism $f:\Sigma'\to \Sigma$ inducing $\varphi$ and $X:=f(\Sigma_1)$ is a flare surface with $X^*=U$.
\end{proof}

A \emph{finite-type exhaustion} $\vX$ is a collection $\{X_n\st n\in \NN\}$ of properly embedded finite-type subsurfaces such that $X_n\subset X_{n+1}$ for every $n$ and $\bigcup_{n\in\NN} X_n=\Sigma$.
 
\subsection{Cantor-Bendixson rank and subsets of the Cantor set}
For the reader's convenience, we collect here some basic facts about the Cantor-Bendixson rank and the Cantor set that will be useful throughout the paper. We refer to \cite{Kechris_Classical} for the definition of the Cantor-Bendixson derivative.

Given a topological space $X$, its \emph{Cantor-Bendixson rank} is the smallest ordinal $\alpha$ such that $X^\alpha=X^{\alpha+1}$, where $X^{\alpha}$ is the $\alpha$-th Cantor-Bendixson derivative of $X$ and $X^0$ is set to be $X$. In particular, if a set is perfect, it has rank $0$. 

A class of topological spaces that will be especially interesting for us are (countable) ordinals of the form
$$X=\omega^\alpha n+1,$$
where $\omega$ is the smallest countable ordinal, $\alpha$ is any countable ordinal and $n$ is a nonnegative integer. Here the topology is the order topology. In terms of the Cantor-Bendixson derivative the pair $(\alpha,n)$ (also called \emph{characteristic system of $N$}) is characterised as follows:
\begin{itemize}
    \item $|X^{(\beta)}|=\infty$ for $\beta<\alpha$,
    \item $|X^{(\alpha)}|=n$,
    \item $|X^{(\beta)}|=0$ for $\beta>\alpha$.
\end{itemize}
In particular, the Cantor-Bendixson rank of such an $X$ is $\alpha+1$.

For a point $x\in X$, we define its \emph{(Cantor-Bendixson) rank} to be the smallest ordinal $\alpha$ such that $x\notin X^{\alpha}$, if such $\alpha$ exists, and $0$ otherwise. With this definition, the highest rank elements in $\omega^\alpha n+1$ have rank $\alpha+1$ and points of the Cantor set have rank $0$.
 
We will implicitly use the following simple topological characterisation of the Cantor set.
\begin{prop}[Brouwer, \cite{Brower_Structure}]
A topological space is a Cantor set if and only if it is non-empty, perfect, compact, totally disconnected and metrisable.
\end{prop}
We will be interested in open and closed subsets of the Cantor set. Perhaps surprisingly, there are only two types of non-empty open subsets of the Cantor set:
\begin{prop}[Gruenhage--Schoenfeld, \cite{SG_Alternate}]
A non-empty open subset of the Cantor set is homeomorphic to either the Cantor set (if it is compact) or to the Cantor set minus a point (if it is not compact).
\end{prop}

On the other hand, there are (uncountably) many types of closed subsets of the Cantor set. Still, we can describe their structure in a relatively accurate way. Indeed, a consequence of Brouwer's result and the Cantor-Bendixson theorem is that a closed subset of the Cantor set is either countable or of the form
$$C\cup N,$$
where $C$ is a closed subset homeomorphic to a Cantor set and $N$ is open and countable. 
The sets $C$ and $N$ are disjoint, but the closure of $N$ might intersect $C$.

Moreover, Mazurkiewicz and Sierpi{\'n}ski \cite{MS_Contribution} showed that countable closed subsets of the Cantor set are exactly the sets homeomorphic to ordinals (with the order topology) of the form
$$\omega^\alpha n+1,$$
where $\omega$ is the smallest countable ordinal, $\alpha$ is any countable ordinal and $n$ is a nonnegative integer.

\section{Arcs and their complements}\label{sec:arcs&complements}
In this section we want to understand some topological properties of arcs and of their complements. We will discuss loops and 2-ended arcs separately.

\subsection{2-ended arcs}
We start by looking at 2-ended arcs.  Note that such arcs are nonseparating since if an arc $\alpha$ joins two ends $e\neq f$, we can find a simple closed curve $\gamma$ separating $e$ from $f$.  Since $\gamma$ intersects $\alpha$ an odd number of times, $\alpha$ must be nonseparating.

The first fact that we want to prove is that if two distinct ends are joined by an arc, they have admissible descending chains of flare surfaces which are especially adapted to the arc. Very informally, we want to say that with respect to a special collection of flare surfaces, the arc goes ``straight out'' towards its ends.

\begin{lemma}\label{lem:straight-out}
Let $\alpha$ be a proper arc joining two distinct ends $e$ and $f$. Then there are admissible descending chains of flare surfaces $U_1\supset U_2\supset \dots$ defining $e$ and $V_1\supset V_2\supset\dots$ defining $f$ such that $i(\alpha, \partial U_i)=i(\alpha, \partial V_i)=1$ for every $i$.
\end{lemma}

\begin{proof}
Fix admissible descending chains of flare surfaces $\{X_n\}$ defining $e$ and $\{Y_n\}$ defining $f$. We will modify the sequence $\{X_n\}$ to obtain flare surfaces as required; the same argument can be applied to $\{Y_n\}$.

As $\alpha$ is proper, it intersects $\partial X_n$ finitely many times. Furthermore, there is $n_1$ such that for every $n\geq n_1$ the curve $\partial X_n$ separates $e$ and $f$. In particular $\partial X_n$, for $n\geq n_1$, must intersect $\alpha$ an odd number of times.

If $i(\partial X_{n_1}, \alpha)=1$, we set $U_1:=X_{n_1}$. Otherwise, consider the subarc $\alpha_1$ of $\alpha$ between the first and last intersection of $\alpha$ with $\partial X_{n_1}$. Choose a small enough regular neighbourhood $N$ of $\alpha_1\cup \partial X_{n_1}$, such that $N\cap\alpha$ is the union of $\alpha_1$ and exactly two extra segments of $\alpha$ (see Figure \ref{fig:smallregnbhood}).

\begin{figure}[ht]
\begin{center}
\includegraphics[width=.6\textwidth]{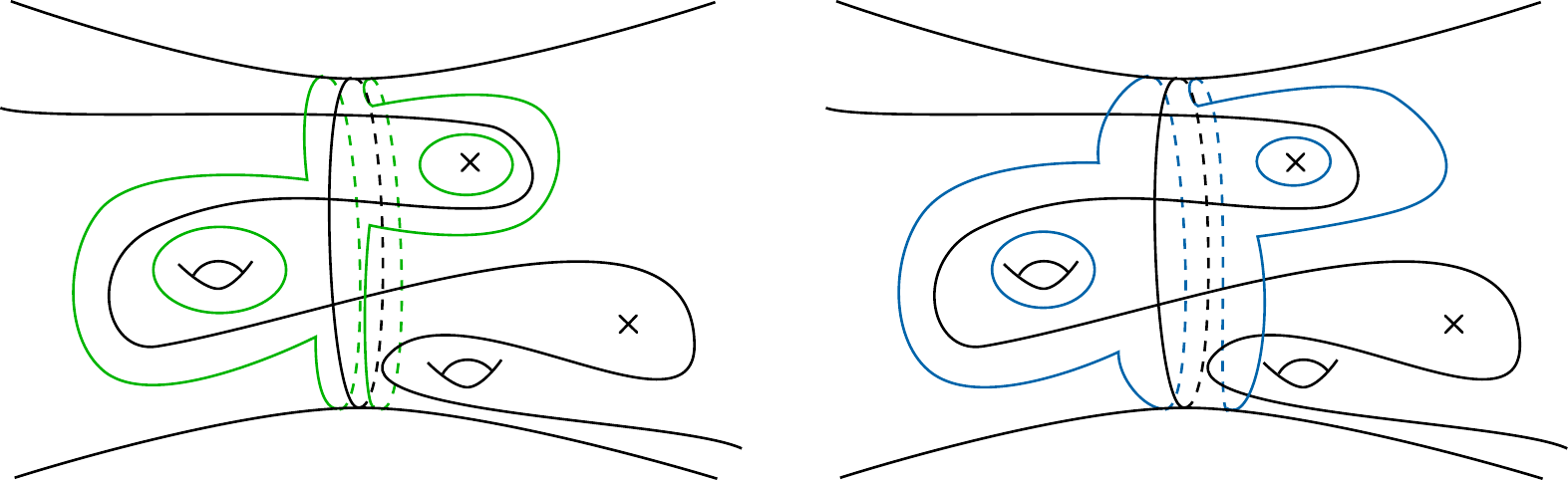}
\caption{An allowed neighbourhood on the left-hand side and one that is not allowed on the right-hand side}\label{fig:smallregnbhood}
\end{center}
\end{figure}

The arc $N\cap\alpha$ joins two distinct boundary components of the finite-type surface $N$ which are separated by $\partial X_{n_1}$. One can construct an arc $\beta$ with the same endpoints as $N\cap\alpha$ and intersecting $\partial X_{n_1}$ once. Furthermore, it is not hard to see that there is a homeomorphism $\phi$ of $N$, fixing its boundary components pointwise, sending $\alpha\cap N$ to $\beta$. Then we can define $U_1$ to be the flare surface with boundary $\phi^{-1}(\partial X_{n_1})$ and with $U_1^*=(X_{n_1})^*$. By construction, $i(\alpha,\partial U_1)=1$.

Since $N$ is compact, there is an $n_2$ such that $X_{n_2}\subsetneq U_1$. With the same argument as before, we can replace $X_{n_2}$ by a flare surface $U_2$ with the same space of ends as $X_{n_2}$ and $i(\partial U_2,\alpha)=1$. Repeating the procedure we obtain the required sequence of flare surfaces.
\end{proof}

We will now show that the mapping class group orbit of a 2-ended arc is determined by the mapping class group orbit of the pair of ends. Since such arcs are nonseparating, one might expect this result to be obvious, but in the next section we will show that this is not true when considering nonseparating loops.

\begin{lemma}\label{lem:orbit-of-arcs}
The mapping class group orbit of an arc $\alpha$ joining two distinct ends $e$ and $f$ is the set of arcs joining ends $e'$ and $f'$, where $(e',f')\in\MCG(\Sigma)\cdot (e,f)$.
\end{lemma}

\begin{proof}
Clearly, if $\beta$ is in the mapping class group orbit of $\alpha$, its pair of endpoints is in the mapping class group orbit of $(e,f)$. So suppose that $\beta$ is an arc with endpoints $(e',f')\in\MCG(\Sigma)\cdot(e,f)$.

Using Lemma \ref{lem:straight-out}, we can construct a compact exhaustion of $\Sigma$ by surfaces $\{A_n\}$ (respectively, $\{B_n\}$) such that for every $n$:
\begin{itemize}
    \item each boundary component of $A_n$ (respectively, $B_n$) is a separating curve,
    \item $A_n$ separates $e$ from $f$ (respectively, $B_n$ separates $e'$ from $f'$),
    \item $\alpha$ (respectively, $\beta$) intersects exactly two boundary components of $A_n$ (respectively, $B_n$) in exactly one point each.
\end{itemize}
We then run the same procedure as in the proof of \cite[Theorem 1]{Richards_Classification}, with the extra condition that at each step we send $\alpha\cap A_n$ to a subarc of $\beta$ or $\beta\cap B_m$ to a subarc of $\alpha$ (depending on the step). This is possible because on a compact surface, two arcs joining distinct boundary components are in the same mapping class group orbit.
The result is a homeomorphism of $\Sigma$ to itself sending $\alpha$ to $\beta$.
\end{proof}

Our next goal is to describe the topology of the complement of an arc joining two distinct ends.

\begin{lemma}\label{lem:arcbtw2ends}
Suppose $\alpha$ is an arc joining two distinct ends $e$ and $f$. Let $\sim$ be the equivalence relation on $\Ends(\Sigma)$ generated by $e\sim f$ and $\pi:\Ends(\Sigma)\to \Ends(\Sigma)/_\sim$ the natural projection. Then $$(\Ends(\Sigma\sm \alpha),\Ends_g(\Sigma\sm \alpha))\simeq (\pi(\Ends(\Sigma)),\pi(\Endsg(\Sigma)))$$
and $\Sigma\sm\alpha$ has the same genus as $\Sigma$.
\end{lemma}

\begin{proof}
We show first that the genus of $\Sigma\sm\alpha$ is the same as the genus of $\Sigma$.
Indeed, using Lemma \ref{lem:straight-out} and properness of the arc, we can write $\Sigma$ as
$$\Sigma=\bigcup_{n\in I}K_n\cup X$$
for some $I\subset \ZZ$, where
\begin{itemize}
    \item the $K_n$ are compact subsurfaces with boundary transverse to $\alpha$ and $X$ is a possibly disconnected subsurface,
    \item if any two subsurfaces in the union intersect, they do so in a union of boundary components,
    \item $\alpha\subset\bigcup_{n\in I}K_n$,
    \item for every $n$, $\alpha$ intersects $\partial K_n$ in exactly two points, belonging to different boundary components.
\end{itemize}

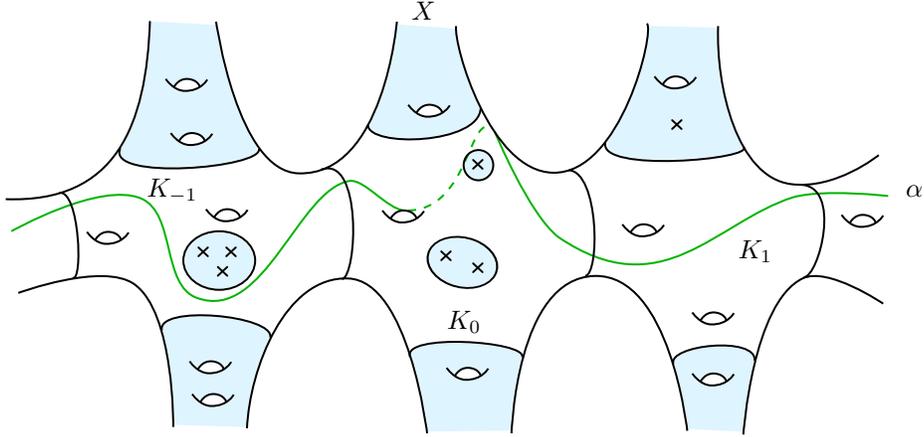
\begin{figure}[ht]
\begin{center}
\begin{overpic}{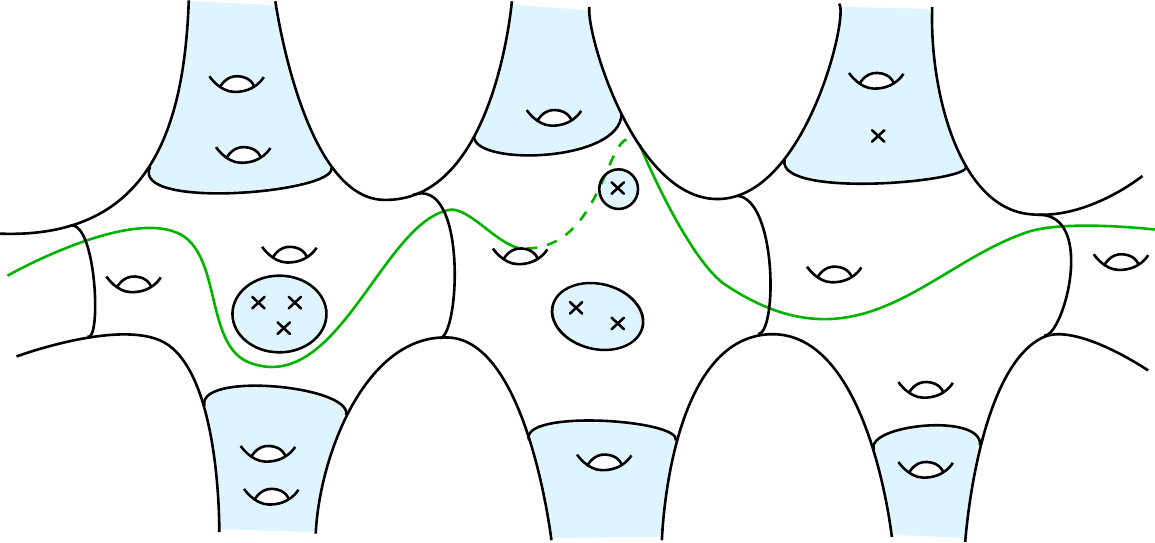}
\put(102,27){$\alpha$}
\put(16,27){$K_{-1}$}
\put(50,12){$K_0$}
\put(83,20){$K_1$}
\put(46,47){$X$}
\end{overpic}
\caption{A decomposition of $\Sigma$ as in the beginning of the proof of Lemma \ref{lem:arcbtw2ends}, where $X$ is shaded}\label{fig:decomposition}
\end{center}
\end{figure}
Then $X\sm\alpha=X$, so the genus of $X\sm\alpha$ is the genus of $X$.
Since $K_n$ is of finite-type, the genus of $K_n\sm\alpha$ is the same as the genus of $K_n$.
Hence the genus of $\Sigma\sm\alpha$ is the same as the genus of $\Sigma$.

To prove the statement about the space of ends of $\Sigma\sm\alpha$, we define a map $$\phi:\Ends(\Sigma)/_\sim\to \Ends(\Sigma\sm\alpha)$$ as follows.

If $[x]\neq [e]$, there is an admissible descending chain of flare surfaces $\{U_n\}$ for $x$ which is disjoint from $\alpha$ (by properness). We set $\phi([x])$ to be the end defined in $\Sigma\sm\alpha$ by the class of $\{U_n\}$.

To define $\phi([e])$, choose two admissible descending chains of flare surfaces $\{U_n\}$ and $\{V_n\}$ for $e$ and $f$ respectively, such that $U_1\cap V_1=\emptyset$. Fix a complete hyperbolic structure on $\Sigma$.
Define
$$W_n:=U_n\cup V_n\cup N_{\frac{1}{n}}(\alpha),$$
where $N_{\frac{1}{n}}(\alpha)$ is the $\frac{1}{n}$-neighbourhood of $\alpha$. We claim that $\{Z_n:=W_n\sm \alpha\}$ is an admissible descending chain and hence defines an end of $\Sigma\sm\alpha$, which we will set to be $\phi([e])$. First of all, it is clear that $Z_n$ is the closure of an unbounded open set with compact boundary. Moreover, since $U_n,V_n$ and $\alpha$ are connected, $W_n$ is connected, and since $\alpha$ joins two different ends, it cannot separate $W_n$, so $Z_n$ is connected. Finally, let $K$ be a compact subset of $\Sigma\sm\alpha$. Then $K$ is a compact subset of $\Sigma$ disjoint from $\alpha$. So there is an index $n_1>0$ such that the distance of $K$ from $\alpha$ is at least $\frac{1}{n_1}$.
Because $\{U_n\}$ and $\{V_n\}$ are admissible chains, there is an index $n_2$ such that $K$ is disjoint from $U_n\cup V_n$ for every $n\geq n_2$. This implies that $K$ is disjoint from $Z_n$, for every $n\geq\max\{n_1,n_2\}$.

Since $\Ends(\Sigma)$ is Hausdorff, it is not difficult to show that $\phi$ is injective. To prove surjectivity, let $s$ be an end of $\Sigma\setminus\alpha$. Then either $s=\phi([e])$, in which case we are done, or $s\neq \phi([e])$, so there is a simple closed curve $\gamma$ of $\Sigma\sm\alpha$ separating $s$ from $\phi([e])$. Since the $Z_n$ define $\phi([e])$, this implies that there is some $n$ such that $\bar{Z_n}$ is disjoint from $\gamma$. Let $\{Y_m\}$ be an admissible descending chain of flare surfaces in $\Sigma\setminus\alpha$ defining $s$; we can also assume that $Y_1$ is disjoint from $\gamma$. Then the $Y_m$ are unbounded in $\Sigma$ as well and define an end $x$ of $\Sigma$ satisfying $\phi([x])=s$.

Continuity follows from the definition of the map. This is enough to show that $\phi$ is a homeomorphism, since $\Ends(\Sigma)/_\sim$ is compact and $\Ends(\Sigma\sm \alpha)$ is Hausdorff.

To conclude, note that $\phi([x])$ is planar if and only if all ends in the equivalence class of $x$ are planar, which implies that $\pi(\Endsg(\Sigma))$ surjects onto the space of nonplanar ends of $\Sigma\sm\alpha$.
\end{proof}

\subsection{Loops}\label{sec:loops}
Loops can be both separating and nonseparating.  What may sound surprising is that nonseparating loops based at ends in the same mapping class group orbit \emph{need not} be in the same mapping class group orbit themselves.  An example is given by the two arcs in Figure \ref{fig:nonsep2orbits}. One can show that $\Sigma\sm\alpha$ is a once-punctured Loch Ness monster, while $\Sigma\sm\beta$ is Jacob's ladder.

\begin{figure}[ht]
\begin{center}
\begin{overpic}[width=.35\textwidth]{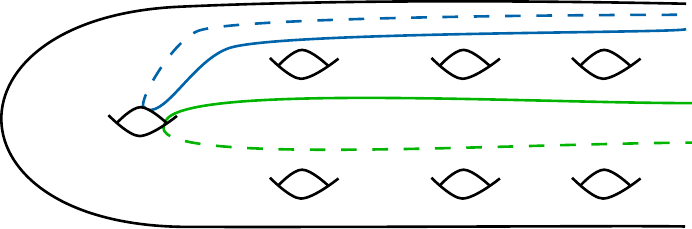}
\put(102,28){$\alpha$}
\put(102,12){$\beta$}
\end{overpic}
\caption{Two different mapping class group orbits of nonseparating loops}\label{fig:nonsep2orbits}
\end{center}
\end{figure}

It is hence impossible to prove a result similar to Lemma \ref{lem:arcbtw2ends} for generic loops. If the loop is separating however, we can give a description of the end space of the components of its complement:

\begin{lemma}\label{lem:separating_arc}
Suppose $\alpha$ is a separating loop based at an end $e$. Let $S$ be a connected component of $\Sigma\sm \alpha$ and set
$$C:=\{x\in\Ends(\Sigma)\st \exists\, X \,\mbox{flare surface of }\Sigma \mbox{ such that } x\in X^*\mbox{  and } X\subset S\}.$$
Then there is a natural homeomorphism $\phi: C\cup\{e\}\to\Ends(S)$. Furthermore, given $f$ in $C$, $\phi(f)$ is planar if and only if $f$ is planar. If $e$ is planar, then $\phi(e)$ is planar.
\end{lemma}

We omit the proof, since it uses the same arguments as that of Lemma \ref{lem:arcbtw2ends}.

Note that if $e$ is nonplanar, its image might be planar: an example is given in Figure \ref{fig:nonptop}, where one of the complementary components of the arc has finite genus (and hence cannot have a nonplanar end).

\begin{figure}[t]
\begin{center}
\includegraphics[width=.4\textwidth]{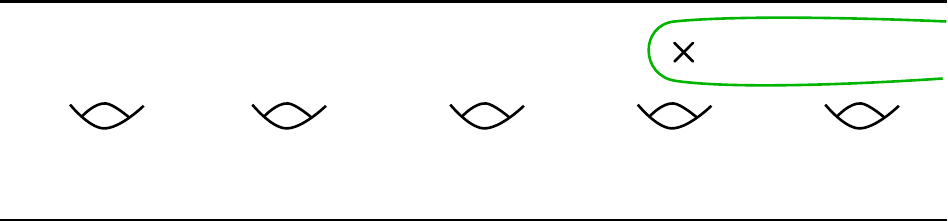}
\end{center}
\caption{Turning a nonplanar end into a planar one}
\label{fig:nonptop}
\end{figure}

\subsection{One-cut (homeomorphic) subsurfaces}
We call a complementary component of a separating loop a \emph{one-cut subsurface}. The goal of this section is to prove Theorem \ref{thmintro:one-cut-homeo}, i.e.\ to show that any infinite-type subsurface admits a \emph{one-cut homeomorphic subsurface} -- a one-cut subsurface which is homeomorphic to the full surface.

Suppose $\alpha$ is an arc joining ends $e,f \in \Ends(\Sigma)$ and $V$ is a flare surface such that $f \in V^*$ and $\abs{\partial V \cap \alpha} = 1$ as in Lemma \ref{lem:straight-out}.  Let $A = \partial V \cup (\alpha \cap \Sigma \sm V)$, and let $N$ be a regular neighbourhood of $A$. Then $\partial N$ is the disjoint union of a closed curve $c$ isotopic to $\partial V$, and a loop $L$ based at $e$. Define the \emph{lasso around $V$ along $\alpha$} to be the loop $L$. Intuitively, the lasso around $V$ along $\alpha$ is the loop formed by starting at $e$, travelling along $\alpha$ until reaching $V$, traversing the simple closed curve $\partial V$, and travelling back along $\alpha$.

\begin{lemma}\label{lem:one-cut-homeo}
If either:
\begin{enumerate}
    \item $\Sigma$ has an isolated nonplanar end, or
    \item $\Sigma$ has a puncture with infinite orbit, or
    \item the space of ends of $\Sigma$ contains a clopen subset homeomorphic to a Cantor set, whose points are either all nonplanar ends or all planar ends,
\end{enumerate}
then $\Sigma$ admits a one-cut homeomorphic subsurface.
\end{lemma}

\begin{proof}
In case 1, there is a separating simple closed curve on $\Sigma$ cutting off a Loch Ness monster with one boundary component. We can choose an arc as in Figure \ref{fig:one-cut-homeo}, left-hand side.

\begin{figure}[ht]
\begin{center}
\begin{overpic}{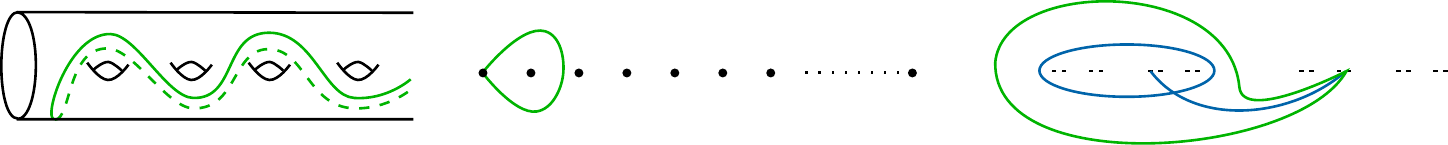}
\put(10,-3){Case 1}
\put(45,-3){Case 2}
\put(33,6){$p_1$}
\put(36,6){$p_2$}
\put(63,6){$e$}
\put(80,-3){Case 3}
\put(82,1){$\alpha$}
\put(92,6){$e_2$}
\put(77,4){$V$}
\end{overpic}
\vspace{.5cm}
\end{center}
\caption{Loops cutting off a one-cut homeomorphic subsurface}
\label{fig:one-cut-homeo}
\end{figure}
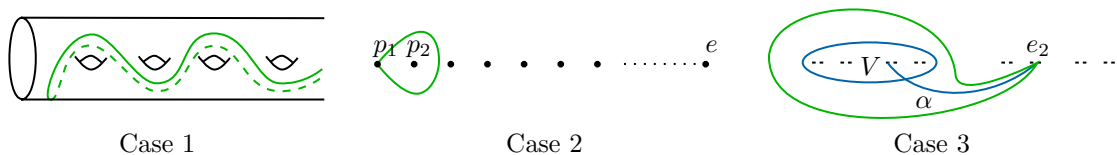

In case 2, we can find a sequence $\{p_n\}$ of punctures in the given infinite orbit which converges to some end $e$. We then choose an arc as in Figure \ref{fig:one-cut-homeo}, in the centre.

In case 3, write the Cantor set $C$ as disjoint union of two Cantor sets $C_1$ and $C_2$. We can find a simple closed curve cutting off a flare surface $V$ with space of ends $C_1$. Pick ends $e_i\in C_i$ and an arc $\alpha$ from $e_1$ to $e_2$ with $i(\alpha, \partial V)=1$. The arc we need is the lasso around $V$ along $\alpha$ (see Figure \ref{fig:one-cut-homeo}, right-hand side). 
\end{proof}
Using this, we show the general result.
\begin{proof}[Proof of Theorem \ref{thmintro:one-cut-homeo}]
If $\Sigma$ is of finite-type, both its genus and its number of planar ends are finite. A one-cut subsurface will have either strictly smaller genus, or strictly fewer punctures (or both), and thus cannot be homeomorphic to $\Sigma$.
Suppose then that $\Sigma$ is of infinite-type. We will show that we can reduce to one of the three cases of Lemma \ref{lem:one-cut-homeo}.

If $\Ends(\Sigma)$ contains infinitely many isolated points, then either there is a puncture of infinite orbit, and we are in case 2, or there is an isolated nonplanar end, and we are in case 1.

Suppose then that $\Ends(\Sigma)$ does not contain infinitely many isolated points. Then $\Ends(\Sigma)=C\sqcup F$, where $C$ is a Cantor set and $F$ is a finite set, possibly empty. Then:
\begin{enumerate}[label=(\alph*)]
\item if $\Endsp(\Sigma)\cap C$ is empty, $C$ is a clopen subset of nonplanar ends and we are in case 3.
\item if $\Endsp(\Sigma) \cap C$ is not empty, by \cite{SG_Alternate} it is either homeomorphic to a Cantor set or to a Cantor set minus a point. In both cases, we can find a clopen subset $C'\subset C$ given by planar ends and homeomorphic to the Cantor set and we are again in case 3.
\end{enumerate}
\end{proof}

\subsection{Homeomorphic subsurfaces from nonseparating arcs}
We now briefly discuss nonseparating arcs, and cases where their complements can or cannot be homeomorphic subsurfaces.  It is due to this more unpredictable behaviour that we restrict our focus to the separating arcs of Theorem \ref{thmintro:one-cut-homeo}.  We omit the proofs of the following examples, we just mention that the claims can be proved with the same techniques as in Section \ref{sec:omnipresent}.

We note first that homeomorphic subsurfaces can indeed occur as the complements of nonseparating arcs.  For example, the complement of any 2-ended arc in the sphere minus a Cantor set is a homeomorphic subsurface.  Next, consider the Loch Ness monster (Figure \ref{fig:nonsep2orbits}) and label the sole end $e$.  Now remove infinitely many points accumulating to $f \neq e$.  Any nonseparating loop based at an end distinct from $e,f$ has a homeomorphic subsurface as its complement.

In contrast to the case of separating arcs, we can also find surfaces that have no nonseparating arcs whose complement is a homeomorphic subsurface.  Consider a surface with at least two, but finitely many, ends, where all ends are nonplanar. We leave it as an exercise for the reader to see that while both 2-ended arcs and nonseparating loops exist on the surface, the complement of any such arc is \emph{not} a homeomorphic subsurface.

%\begin{prop}
%Suppose not all ends of $\Sigma$ have finite mapping class group orbit. Then there is an arc $\alpha$ joining distinct ends whose complement is homeomorphic to $\Sigma$.
%\end{prop}
%\begin{proof}
%Suppose first that $\Ends(\Sigma)$ contains infinitely many isolated points. Then it must contain infinitely many isolated ends of a given type (planar or nonplanar) and any arc joining two such ends has complement homeomorphic to $\Sigma$.

%If $\Ends(\Sigma)$ does not contain infinitely many isolated points, then it is a Cantor set $C$ union a finite (and possibly empty) set. Note that $C$ is clopen in the space of ends. Furthermore, $C\cap \Endsp(\Sigma)$ is an open subset of $C$ and of $\Ends(\Sigma)$. If it is empty, then $C\subset \Endsg(\Sigma)$. Otherwise, it must be homeomorphic to either a Cantor set or a Cantor set minus a point \cite{SG_Alternate}. In both cases, we can find a clopen subset $C'\subset C$ which is entirely contained in either $\Endsp(\Sigma)$ or $\Endsg(\Sigma)$. Any arc joining two ends in $C'$ has complement homeomorphic to $\Sigma$.
%\end{proof}

\section{Omnipresent arcs}\label{sec:omnipresent}

In the previous section we have seen that all infinite-type surfaces admit one-cut homeomorphic subsurfaces. This gives us a way to select a subclass of arcs, those which do not avoid any such subsurface. In a sense, we want to think of these arcs as ``truly essential'' ones.
\begin{defn}
Let $\alpha$ be a proper arc joining two distinct ends. We say $\alpha$ is \emph{omnipresent} if it intersects all one-cut homeomorphic subsurfaces $S$ (i.e.\ there are no disjoint representatives of the arc and the subsurface).
\end{defn}

\begin{eg}
Let $\Sigma$ be the sphere minus a Cantor set, and let $\alpha$ be a proper arc with ends $e$ and $f$. Let $V \subset \Sigma$ be a flare surface such that $f \in V^*$, $e \notin V^*$ and $i(\partial V, \alpha) = 1$ (such a surface exists by Lemma \ref{lem:straight-out}). Let $L$ be the lasso around $V$ along $\alpha$, which is a separating loop based at $e$. Let $S$ be the component of $\Sigma \sm L$ such that $S \cap \alpha = \emptyset$. Then since $V^*$ is a clopen subset of the Cantor set, $\Ends(\Sigma)\sm V^* \simeq \Ends(\Sigma)$ so by Lemma \ref{lem:separating_arc}, $S$ is a one-cut homeomorphic subsurface of $\Sigma$. We may conclude no 2-ended arcs in $\Sigma$ are omnipresent.
\end{eg}

\begin{eg}
Suppose $\Sigma$ has finitely many ends. Let $\alpha$ be a proper arc with distinct ends, and suppose $S$ is a one-cut subsurface such that $\alpha \subset \Sigma \sm S$. Since $\Sigma$ has finitely many ends and by Lemma \ref{lem:separating_arc} we know that $\abs{\Ends(S)} < \abs{\Ends(\Sigma)}$, $S$ cannot be homeomorphic to $\Sigma$. Therefore every 2-ended arc is omnipresent. As we will see, this also follows from Proposition \ref{prop:fends-are-shelters} below.
\end{eg}

The two examples above are extreme cases; in general omnipresent arcs form a proper subset of the set of 2-ended arcs. In this section we give a sufficient condition for a 2-ended arc to be an omnipresent arc (Proposition \ref{prop:fends-are-shelters}). For this, we need a preliminary lemma.

\begin{lemma} \label{lem:orbit-characterisation}
Let $e,f \in \Ends(\Sigma)$ be distinct ends. Then $e$ and $f$ are in the same mapping class group orbit if and only if there exists disjoint clopen neighbourhoods $e \in U, f \in V$ and a homeomorphism $\phi:(U,U \cap \Ends_g(\Sigma))\to (V,V\cap \Ends_g(\Sigma))$ such that $\phi(e) = f$.
\end{lemma}
\begin{proof}
Suppose $U \ni e$ and $V \ni f$ are disjoint clopen neighbourhoods and $\phi:(U,U \cap \Ends_g(\Sigma)) \to (V, V \cap \Ends_g(\Sigma))$ is a homeomorphism such that $\phi(e) = f$. Then $U \cup V$ is a clopen set. Define $\eta:U \cup V \to U \cup V$ by
\[
\eta(x) = \begin{cases}
\phi(x) &\text{if } x \in U \\
\phi^{-1}(x) &\text{if } x \in V.
\end{cases}
\]
Then $\eta$ is a homeomorphism of the clopen set $U \cup V$, so we can extend it by the identity to a homeomorphism $\Phi:(\Ends(\Sigma),\Ends_g(\Sigma)) \to (\Ends(\Sigma),\Ends_g(\Sigma))$. Note that $\Phi(e) = f$.

Conversely, suppose $\Phi(e) = f$ for some homeomorphism $\Phi \in \Homeo(\Ends(\Sigma),\Ends_g(\Sigma))$. Let $W \ni e$ be a clopen neighbourhood such that $f \notin W$. Let $V = \Phi(W)\sm W$, which is a clopen neighbourhood of $f$. Let $U = \Phi^{-1}(V)$. Note that $U \subset W$ so $U \cap V = \emptyset$. Furthermore, $\Phi|_U:(U,U\cap F) \to (V,V \cap F)$ is a homeomorphism such that $\Phi|_U(e) = f$.
\end{proof}

We are now in a position to prove one direction of Theorem \ref{thmintro:omnipresent&fends}.

\begin{defn}
An end $e \in \Ends(\Sigma)$ is a \emph{finite-orbit end} if it has finite mapping class group orbit.
\end{defn}

\begin{prop}\label{prop:fends-are-shelters}
Let $\alpha$ be a proper arc in $\Sigma$ with distinct endpoints $e$ and $f$. If $e$ and $f$ are both finite-orbit ends, then $\alpha$ is an omnipresent arc.
\end{prop}
\begin{proof}
Suppose $\alpha$ is not omnipresent, and that $e$ and $f$ are both finite-orbit ends. Let $S \subset \Sigma$ be a one-cut homeomorphic subsurface such that $\alpha \cap S = \emptyset$ and let $\varphi:\Sigma\to S$ be a homeomorphism. Recall from Lemma \ref{lem:separating_arc} the subspace $C \subset \Ends(\Sigma)$ given by
\[
C:= \{x \in \Ends(\Sigma) \mid \exists X \text{ a flare surface of $\Sigma$  such that $x \in X^*$ and $X \subset S$}\}.
\]
Note that both ends $e$ and $f$ are not in $C$, since then $\alpha \cap S \neq \emptyset$. By Lemma \ref{lem:separating_arc} identify $\Ends(S)$ with the subspace $C \cup\{b\} \subset \Ends(\Sigma)$ where $b$ is the base of the separating loop defining $S$.

To ease notation, let $O = \MCG(\Sigma)\cdot e \cup \MCG(\Sigma) \cdot f$, which is a finite set since $e$ and $f$ are finite-orbit ends. Since $\abs{\Ends(S) \sm C} = 1$, we must have $\abs{\varphi(O) \cap C} \geq \abs{O}-1$. Since $e,f \notin C$, $\abs{O \cap C} \leq \abs{O} - 2$. Therefore there exists an end $x \in O$ such that $\varphi(x) \in C$ and $\varphi(x) \notin O$. Let $V \subset S$ be a flare surface such that $\varphi(x) \in V^*$, and note that $V^*$ is a clopen subset of both $\Ends(S)$ and $\Ends(\Sigma)$.

Let $W \subset V^*$ be a clopen neighbourhood of $\varphi(x)$ such that $x \notin W$ and let $Z = \varphi^{-1}(W) \sm W$. Note that $\varphi(Z)\subset V^*\subset C$ and $\varphi(Z) \cap Z = \emptyset$. For all $z \in \varphi(Z)$, $z$ is planar in $\Ends(S)$ if and only if $z$ is planar in $\Ends(\Sigma)$ by Lemma \ref{lem:separating_arc}. Therefore $\varphi|_Z$ induces a homeomorphism of pairs $(Z,Z \cap \Ends_g(\Sigma)) \cong (\varphi(Z),\varphi(Z) \cap \Ends_g(\Sigma))$ sending $x$ to $\varphi(x)$. We are now forced to conclude $x$ and $\varphi(x)$ are in the same $\MCG(\Sigma)$-orbit by Lemma \ref{lem:orbit-characterisation}, a contradiction. 
\end{proof}

\section{Stability conditions}\label{sec:stability}

After playing with a few examples of infinite-type surfaces, one may be tempted to guess that the converse to Proposition \ref{prop:fends-are-shelters} is also true. As the next example shows, this turns out to be wishful thinking.
\begin{eg}\label{eg:unstable}
Consider the sphere with a Cantor set removed, pictured as the tubular neighbourhood of an infinite rooted binary tree without leaves. Remove pairwise disjoint disks, one corresponding to each vertex different from the root, not accumulating anywhere inside the surface. Glue in the place of each disk corresponding to a vertex at distance $n$ from the root the surface $X_n$ with one boundary component and end space given by
$$(\Ends(X_n),\Endsg(X_n))\simeq (\omega^n+1,\{\omega^n\}).$$
Less formally, the $2^n$ copies of $X_n$ each have a single nonplanar end, of rank $n+1$, accumulated by planar ends. Let $\Sigma$ be the surface obtained this way, see Figure \ref{fig:GMCT}.

\begin{figure}[ht]
    \centering
    \includegraphics[width=\textwidth]{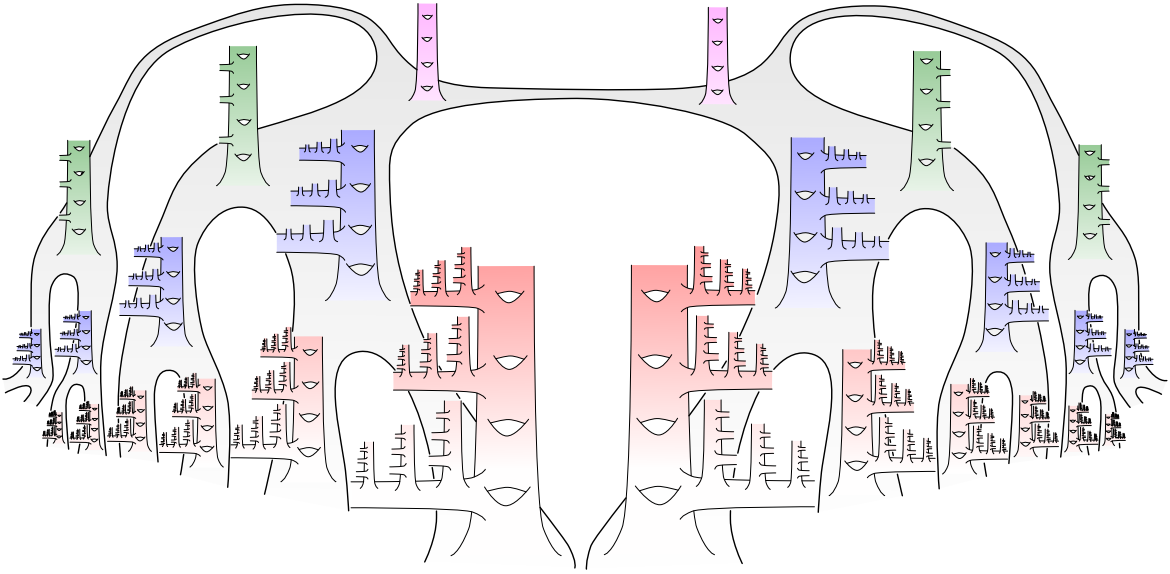}
    \caption{A sketch of the surface $\Sigma$ constructed in Example \ref{eg:unstable}.  The pink, green, blue, and red subsurfaces correspond to copies of $X_0$, $X_1$, $X_2$, and $X_3$ respectively.}
    \label{fig:GMCT}
\end{figure}
It is not hard to see that the finite-orbit ends of $\Sigma$ are the nonplanar ends of finite non-zero rank (i.e.\ the nonplanar ends of the surfaces $X_n$). On the other hand, one can show that any arc $\alpha$ joining two ends $e, f$ of rank zero is an omnipresent arc. Indeed, for any loop $\beta$ such that $\Sigma\sm\beta=S_1\cup S_2$, where $S_1$ contains $\alpha$, either $e$ or $f$ is completely contained in $S_1$ -- say it's $e$. Since there is a sequence of finite-orbit ends converging to $e$, there are infinitely many finite-orbit ends which are not completely contained in $S_2$. It is then easy to argue (using Lemmas \ref{lem:separating_arc} and \ref{lem:orbit-characterisation}) that $S_2$ cannot be homeomorphic to $\Sigma$.
\end{eg}

\subsection{Stable ends and a preorder}
The example above shows that to prove a converse to Proposition \ref{prop:fends-are-shelters} we need to restrict to a good class of surfaces.  Otherwise, we can encounter problems given by ends with wild behavior; similar to the situation in Mann and Rafi's work \cite{MR_Large}. There the authors introduce the following definition:
\begin{defn}
Let $e$ be an end of a surface. A clopen neighbourhood $U$ of $e$ is \emph{stable} if for every clopen neighbourhood $V\subset U$ there is a clopen neighbourhood $U'\subset V$ such that $(U,U\cap\Ends_g(\Sigma))\simeq(U',U'\cap\Ends_g(\Sigma))$. We say that an end is \emph{stable} if it admits a stable neighbourhood, and \emph{unstable} otherwise.
\end{defn}
As shown in \cite[Lemma 4.17]{MR_Large}, an end is stable if and only if there exists a clopen neighbourhood $U$ of $e$ such that for any clopen set $V \subset \Ends(\Sigma)$ with $e \in V \subset U$, $(V,V \cap \Ends_g(\Sigma)) \simeq (U,U\cap \Ends_g(\Sigma))$. We will use this fact without mention in the rest of the paper. We also record for later use the following consequence of \cite[Lemmas 4.16 and 4.17]{MR_Large}.
\begin{lemma}\label{lem:orbit-stable-nhoods}
Let $\Sigma$ be an infinite-type surface, and suppose $e \in \Ends(\Sigma)$ is a stable end. Let $f \in \MCG(\Sigma)\cdot e$. Then $f$ is a stable end and if $V \ni e$ and $W \ni f$ are stable neighbourhoods, then $(V,V \cap \Ends_g(\Sigma)) \simeq (W,W \cap \Ends_g(\Sigma))$. In particular, two stable neighbourhoods of the same end are homeomorphic.
\end{lemma}

Using the stability condition of an end, we define a class of ``well behaved'' surfaces:
\begin{defn}
A surface is \emph{stable} if all of its ends are stable.
\end{defn}

\begin{rmk}\label{rmk:alot-of-stables}
Even if we restrict to stable surfaces, we are still considering a large class: there are uncountably many such surfaces up to homeomorphism (so the set has cardinality of the continuum, just like the set of homeomorphism classes of surfaces). Indeed, it is for instance not hard to see that if a surface has either only planar or only nonplanar ends and the space of ends is countable, then it is stable, and it is an instructive exercise for the reader to prove this fact. As there are uncountably many countable ordinals and $\omega^\alpha+1$ is homeomorphic to $\omega^\beta+1$ if and only if $\alpha=\beta$, we deduce from this the statement about the cardinality of the set of stable surfaces.
\end{rmk}

The surface constructed in Example \ref{eg:unstable} is not stable: all the rank-zero ends are unstable. This is because they are all limits of sequences of finite-orbit ends, which gives an obvious obstruction to stability:
\begin{lemma}\label{lem:unstable-fend-accumulation}
Let $e \in \Ends(\Sigma)$ be a limit point of the set of finite-orbit ends. Then $e$ is not stable.
\end{lemma}

\begin{proof}
Suppose $e$ is a stable end with stable neighbourhood $V$. Towards a contradiction, let $f$ be a finite-orbit end such that $f\in V$ and $e \notin \MCG(\Sigma)\cdot f$. Let $U$ be a clopen set such that $e \in U \subset V$ and $U \cap \MCG(\Sigma)\cdot f = \emptyset$. Since $V$ is stable, there is a homeomorphism $\phi:(V,V \cap \Ends_g(\Sigma)) \to (U,U \cap \Ends_g(\Sigma))$. Let $W \subset V$ be a clopen neighbourhood of $f$ such that $W \cap U = \emptyset$. Then $\phi(W) \cap W = \emptyset$, so by Lemma \ref{lem:orbit-characterisation}, $\phi(f) \in \MCG(\Sigma)\cdot f$, a contradiction.
\end{proof}

\begin{rmk}
The same proof of Lemma \ref{lem:unstable-fend-accumulation} shows that an end $e$ is unstable also if it is a limit point of the set of ``local finite-orbit ends with respect to $e$'' -- ends $f$ such that there exists a clopen neighbourhood of $e$ containing only finitely many ends in the mapping class group orbit of $f$. It would be interesting to know if this is the only obstruction to stability.
\end{rmk}

Surfaces can then be more or less wild depending on which ends are allowed to be unstable. To distinguish different types of ends, we will use the preorder $\sqleq$ introduced in \cite{MR_Large}. We recall the definition here.
\begin{defn}
Let $e,f$ be ends of a surface. We say that $e\sqleq f$ if for every clopen neighbourhood $U_f$ of $f$ there is a clopen neighbourhood $U_e$ of $e$ and a clopen subset $U_f'\subset U_f$ such that $(U_e,U_e\cap\Endsg(\Sigma))\simeq(U_f',U_f'\cap\Endsg(\Sigma))$.
\end{defn}
\begin{rmk}\label{rmk:order&orbits}
It is not hard to show that $e\sqleq f$ if and only if $f$ is in the closure of the mapping class group orbit of $e$. Therefore, in general $\sqleq$ is not a partial order, since if $e$ and $f$ are distinct ends belonging to the same mapping class group orbit, $e\sqleq f$ and $f\sqleq e$. It is unclear if this is the only obstruction to $\sqleq$ being a partial order, that is, if it induces a partial order on the set of mapping class group orbits of ends.
\end{rmk}

We define $\sim$ to be the equivalence relation on $\Ends(\Sigma)$ given by $e\sim f$ if $e\sqleq f$ and $f\sqleq e$. Then $\sqleq$ induces a partial order, which we also denote by $\sqleq$, on the quotient $\Ends(\Sigma)/\sim$.

\begin{defn}
We say that an end is \emph{maximal} (respectively, an \emph{immediate predecessor of a maximal end}) if its class in $\Ends(\Sigma)/\sim$ is maximal (respectively, immediate predecessor of a maximal element) with respect to $\sqleq$.
\end{defn}

\begin{rmk}\label{rmk:fends-are-maximal}
Using the characterisation of the preorder in Remark \ref{rmk:order&orbits}, we deduce that finite-orbit ends are maximal, since their orbits are already closed in $\Ends(\Sigma)$.
\end{rmk}

Using these notions, Mann and Rafi introduced in \cite{MR_Large} the class of tame surfaces: a surface is \emph{tame} if every end which is either maximal or an immediate predecessor of a maximal end is stable.

Clearly, stable surfaces are tame, but the converse doesn't hold, as the following example shows.

\begin{eg}\label{eg:tame-not-stable}
Consider $X=\RR^2\sm\bigcup_{n=1}^\infty D_n$, where $D_n$ is the disk centered at $(n,0)$ of radius $\frac{1}{3}$. Glue in the place of $D_n$ a copy of $X$ along one of its boundary components to get a surface $Y$. Finally, glue to each boundary component of $Y$ a copy of the surface $S$ constructed in Example \ref{eg:unstable} with a disk removed, to get a surface $\Sigma$. Maximal and immediate predecessors of maximal ends of $\Sigma$ are those corresponding to the ends of $Y$, which are stable. Thus $\Sigma$ is tame. But, again since stability is a local property and $S$ has unstable ends, $\Sigma$ is not stable.
\end{eg}
\begin{figure}[ht]
    \centering
    \includegraphics{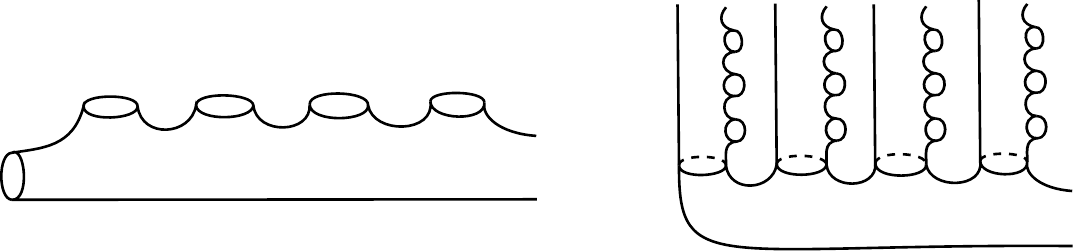}
    \caption{The surfaces $X$ (left-hand side} and $Y$ (right-hand side) defined in Example \ref{eg:tame-not-stable}
    \label{fig:flutes}
\end{figure}

\subsection{Omnipresent arcs on stable surfaces}
The main goal of this section is to show that the converse of Proposition \ref{prop:fends-are-shelters} holds for the class of stable surfaces (Theorem \ref{thmintro:omnipresent&fends}). We begin with the following property of surfaces whose maximal ends are stable.

\begin{prop}\label{prop:finitefends}
Let $\Sigma$ be a surface whose maximal ends are stable. Then $\Sigma$ has finitely many finite-orbit ends.
\end{prop}
\begin{proof}
Suppose $\Sigma$ has infinitely many finite-orbit ends. Then there is an end $e$ which is a limit point of the set of finite-orbit ends and hence is unstable, by Lemma \ref{lem:unstable-fend-accumulation}. By the assumption on $\Sigma$, $e$ cannot be maximal, but by \cite[Proposition 4.7]{MR_Large}, there is a maximal $f$ such that $f\sqgeq e$ (and $f\not\sim e$). So (by Remark \ref{rmk:order&orbits}) there is a sequence of elements in the mapping class group orbit of $e$ converging to $f$. But as all elements in the sequence are limit points of the set of finite-orbit ends, so is $f$. Thus $f$ is unstable, a contradiction.
\end{proof}

The surface constructed in Example \ref{eg:unstable} has infinitely many finite-orbit ends (all the nonplanar ends of finite rank), and indeed some of its maximal ends are unstable. Moreover, the condition in Proposition \ref{prop:finitefends} is strictly weaker than being tame, as the next example shows.

\begin{eg}
Let $S$ be the surface constructed in Example \ref{eg:unstable} with a disk removed. Consider the one-ended surface $F$ with countably infinitely many boundary components converging to the end (a flute surface). Glue a copy of $S$ to each boundary component of $F$ to get a surface $\Sigma$. There is a unique maximal end of $\Sigma$ (which informally corresponds to the unique end of $F$) and it is stable. Its immediate predecessors are those corresponding to the maximal ends in each copy of $S$, and among these, the rank-zero ones are unstable. So $\Sigma$ is not tame.
\end{eg}

The next result is the key lemma to prove Theorem \ref{thmintro:omnipresent&fends}.
\begin{lemma}\label{lem:dispensible-stable-nhood}
Let $e \in \Ends(\Sigma)$ be a stable end with infinite mapping class group orbit. Then there exists a stable clopen neighbourhood $V \ni e$ such that for all clopen neighbourhoods $U\subset V$ of $e$,  $(\Ends(\Sigma)\sm U,\Ends_g(\Sigma) \sm U) \simeq (\Ends(\Sigma),\Ends_g(\Sigma))$.
\end{lemma}

\begin{proof}
Let $Y = \MCG(\Sigma)\cdot e$. Fix an accumulation point $f \neq e$ of $Y$, which exists since $Y$ is infinite and $\Ends(\Sigma)$ is compact. Let $\{W_1,W_2,\ldots\}$ be a sequence of clopen sets defining $f$, that is, $W_i \supset W_{i+1}$ for all $i \geq 1$, and $\bigcap_i W_i = \{f\}$. Let $e_0 = e$, and let $V_0 \ni e$ be a stable neighbourhood of $e$ such that $f \notin V_0$. Now inductively define $e_i \in Y$ and $V_i \subset W_i$ such that $V_i$ is a stable neighborhood of $e_i$, disjoint from $\left(\bigcup_{j=0}^{i-1}V_j\right)$ and $f\notin V_i$, as follows. Fix $i \geq 1$ and let $Z_i = W_i \sm \left(\bigcup_{j=0}^{i-1}V_j\right)$. $Z_i$ is a clopen neighbourhood of $f$,
so $Y \cap Z_i \neq \emptyset$. Choose $e_i \in Z_i\cap Y$ distinct from $f$, and let $V_i$ be a clopen neighbourhood of $e_i$ such that $V_i \subset Z_i$ and $f \notin V_i$. By intersecting with a stable neighbourhood of $e_i$, we may choose $V_i$ to be a stable neighbourhood of $e_i$.

Let $V:= V_0$ and $U\subset V$ be a clopen neighbourhood of $e$, so $U$ is a stable neighbourhood of $e$. By construction, each $V_i$ is a stable neighbourhood of $e_i$, $U \cap V_i = \emptyset$ for all $i \geq 1$, and $V_i \cap V_j = \emptyset$ for all $i \neq j$. By Lemma \ref{lem:orbit-stable-nhoods}, there exist homeomorphisms $\phi_0:(U,U \cap \Ends_g(\Sigma)) \to (V_1,V_1 \cap \Ends_g(\Sigma))$ and $\phi_i:(V_i,V_i \cap \Ends_g(\Sigma)) \to (V_{i+1},V_{i+1} \cap \Ends_g(\Sigma))$ for all $i \geq 1$. Define a map $\Phi:\Ends(\Sigma) \to \Ends(\Sigma) \sm U$ by
\[
\Phi(x) = \begin{cases}

\phi_0(x) &\text{if } x \in U,\\
\phi_i(x) &\text{if } x \in V_i \text{ for some } i \geq 1, \\
x &\text{otherwise.}
\end{cases}
\]
The map $\Phi$ is a bijection, and it restricts to a bijection $\Phi|_{\Ends_g(\Sigma)}:\Ends_g(\Sigma) \to \Ends_g(\Sigma)\sm U$. Since $\Ends(\Sigma)$ is compact and $\Ends(\Sigma)\sm U$ is Hausdorff, it suffices to show $\Phi$ is continuous. We will show sequential continuity since $\Ends(\Sigma)$ is metrisable.

Let $A = \{a_1,a_2,\ldots\}$ be a sequence with infinitely many distinct points in $\Ends(\Sigma)$ converging to $a \in \Ends(\Sigma)$. We split the argument into 3 cases. 

{\bf Case 1:} If $\abs{A \cap (U\cup\bigcup_{i=i}^\infty V_i)}$ is finite, then $a \notin U\cup\bigcup_{i=i}^\infty V_i$.
Therefore all but finitely many of the $a_j$ are such that $\Phi(a_j) = a_j$, so $$\Phi\left(\lim_{j\to \infty} a_j\right) = \Phi(a) = a = \lim_{j\to\infty}\Phi(a_j).$$

{\bf Case 2:} If $\abs{A \cap V_i} = \infty$ for some $i \geq 1$, then $a\in V_i$ (as it's a limit of points in $V_i$). Thus all but finitely many $a_j$ belong to $V_i$ and 
$$\Phi\left(\lim_{j\to\infty}a_j\right) = \phi_i\left(\lim_{j\to\infty}a_j\right) = \lim_{j\to\infty}\phi_i(a_j) = \lim_{j\to\infty}\Phi(a_j).$$
If $|A\cap U|=\infty$, the argument is analogous. 

{\bf Case 3:} Finally, suppose $\abs{A\cap U}$ is finite, $\abs{A \cap V_i}$ is finite for all $i \geq 1$ and $\abs{A \cap (U\cup\bigcup_{i=i}^\infty V_i)}$ is infinite.
Since all but finitely many of the $V_i$ are contained in any particular $W_k$, it follows that for each $W_k$, there is some $a_j \in A \cap W_k$.
Therefore $\lim_{j \to \infty} a_j = f$. By the definition of $\Phi$, we have that $\abs{\Phi(A) \cap V_i}$ is finite for all $i\geq 0$, and $\abs{\Phi(A) \cap (U \cup \bigcup_{i=1}^\infty V_i)}$ is infinite. Therefore by the same argument that showed $\lim_{i \to \infty}a_i = f$ we have $\lim_{i \to \infty} \Phi(a_i) = f$. Observing that $\Phi(f) = f$ completes the proof.
\end{proof}

We now have all the ingredients to prove Theorem \ref{thmintro:omnipresent&fends}, the characterisation of omnipresent arcs for stable surfaces. We will actually prove it for a slightly larger class of surfaces.

\begin{thm}\label{thm:stable-omnipresent}
Let $\Sigma$ be a surface such that all ends with infinite orbit are stable. Then an arc is omnipresent if and only if it joins finite-orbit ends.
\end{thm}

\begin{proof}
By Proposition \ref{prop:fends-are-shelters}, if $e$ and $f$ are both finite-orbit ends then $\alpha$ is omnipresent. For the converse, suppose $e$ has infinite mapping class group orbit. By assumption, $e$ is stable. Let $V\subset \Ends(\Sigma)$ be a neighbourhood of $e$ with the property that for all clopen sets $U$ such that $e \in U \subset V$, $(\Ends(\Sigma)\sm U,\Ends_g(\Sigma) \sm U) \simeq (\Ends(\Sigma),\Ends_g(\Sigma))$, which exists by Lemma \ref{lem:dispensible-stable-nhood}. Let $W$ be a flare surface with the following properties:
\begin{enumerate}
    \item $e \in W^*\subset V$
    \item $i(\partial W, \alpha) = 1$, and
    \item if the genus of $\Sigma$ is finite, then the genus of $W$ is 0.
\end{enumerate}
The existence of such a flare surface is guaranteed by Lemma \ref{lem:straight-out}. Let $\gamma$ be the lasso around $W$ along $\alpha$, and let $S$ be the component of $\Sigma \sm \gamma$ such that $S \cap W = \emptyset$. By construction $S \cap \alpha = \emptyset$, so it suffices to show $S$ is homeomorphic to $\Sigma$.

Let $\Pi = \Sigma \sm W$. By the construction of $S$, we have that $S$ is homeomorphic to $\Pi \sm \alpha$. Since $\partial W$ is a simple closed curve, the space of ends of $\Pi$ is obtained by replacing the clopen set $W^*$ with a puncture $p$. More formally,
\[
(\Ends(\Pi),\Ends_g(\Pi)) \cong (\Ends(\Sigma) \sm W^* \cup \{p\},\Ends_g(\Sigma) \sm W^*).
\]
Now $\alpha \cap \Pi$ is a proper arc with endpoints $f$ and $p$. Since $p$ is a puncture, 
\[
(\Ends(\Pi\sm \alpha),\Ends_g(\Pi\sm \alpha)) \simeq (\Ends(\Sigma) \sm W^*,\Ends_g(\Sigma)\sm W^*)
\]
by Lemma \ref{lem:arcbtw2ends}. Therefore $(\Ends(S),\Ends_g(S)) \cong (\Ends(\Sigma) \sm W^*,\Ends_g(\Sigma)\sm W^*) \cong (\Ends(\Sigma),\Ends_g(\Sigma))$ since $W^* \subset V$.

If the genus of $\Sigma$ is finite, then the genus of $W$ is zero and the genus of $\Pi$ is equal to the genus of $\Sigma$. Therefore by Lemma \ref{lem:arcbtw2ends}, the genus of $S$ is equal to the genus of $\Pi$, completing the proof.
\end{proof}

We end this section with an example showing that, as claimed, the hypothesis in Theorem \ref{thm:stable-omnipresent} is weaker than stability.

\begin{eg}
Consider $\RR^2\sm\bigcup_{n=1}^\infty D_n$, where $D_n$ is the disk centered at $(n,0)$ of radius $\frac{1}{3}$. Glue in the place of $D_n$ a copy of the surface $X_n$ defined in Example \ref{eg:unstable}. Then the only unstable end is the unique end of infinite rank (corresponding to the point at infinity of the plane), so all infinite-orbit ends are stable, though the surface itself is not stable. In this example, it is easily checked that finite-orbit ends are precisely the nonplanar ends, so an arc is omnipresent if and only if it joins two nonplanar ends. 
\end{eg}

\section{Graphs of arcs with infinite intersection}\label{sec:arcgraphs}
Recall from the introduction that we are interested in two subgraphs of the arc graph $\vA(\Sigma)$; the graph whose vertices are homotopy classes of arcs, and edges correspond to having disjoint representatives. The first is the \emph{omnipresent arc graph} $\Omega(\Sigma)$, i.e.\ the subgraph spanned by omnipresent arcs. The second is the \emph{two-ended arc graph with endpoints in $P$}, denoted $\vA_2(\Sigma,P)$, where $P$ is a finite set of ends. This is the full subgraph of $\vA(\Sigma)$ spanned by vertices corresponding to all isotopy classes of arcs connecting two distinct ends in $P$. We stress that we allow arcs joining planar and/or nonplanar ends, while so far in the literature only arcs joining planar ends have been considered.  In particular, this allows for two arcs to have infinite intersection.

Note that Theorem \ref{thm:stable-omnipresent} tells us that, in the case of stable surfaces, the omnipresent arc graph is the two-ended arc graph $\vA_2(\Sigma,\vF)$, where $\vF$ is the set of finite-orbit ends. We also note that $\vF$ is finite (by Proposition \ref{prop:finitefends}) and mapping class group invariant.
We therefore restrict our focus to $\vA_2(\Sigma,P)$ and we show:

\begin{thm}\label{thm:2-ended-arc-graph}
Let $\Sigma$ be a surface, $P$ a finite collection of ends of cardinality at least three.  Then:
\begin{enumerate}

    \item $\vA_2(\Sigma,P)$ is connected,
    \item $\vA_2(\Sigma,P)$ is $\delta$-hyperbolic, and the constant $\delta$ can be chosen independently of $\Sigma$ and $P$,
    \item $\vA_2(\Sigma,P)$ has infinite diameter,
    \item if $P$ is mapping class group invariant, $\MCG(\Sigma)$ acts with unbounded orbits on $\vA_2(\Sigma,P)$. The action is continuous if and only if all ends in $P$ are punctures.
\end{enumerate}
\end{thm}

The theorem is proved over Proposition \ref{prop:discontinuous-action}, Corollary \ref{cor:connected}, Lemma \ref{lem:gglapplies}, and Lemma \ref{lem:unbounded}.

\begin{rmk}
Note that requiring $\Sigma$ to be stable is equivalent to requiring $\Sigma$ to satisfy the hypotheses of Theorem \ref{thm:stable-omnipresent} and Proposition \ref{prop:finitefends}, i.e.\ asking that infinite-orbit and maximal ends are stable. This is because, as observed in Remark \ref{rmk:fends-are-maximal}, finite-orbit ends are maximal.
\end{rmk}

Throughout this section, given a finite collection of ends $P$ and a properly embedded subsurface $X$, we say that a boundary curve of $X$ \emph{corresponds to $p\in P$} if it is separating and cuts off a flare surface $Y$ with $\Ends(Y)\cap P=\{p\}$. We say that $X$ has \emph{boundary separating $P$} if its boundary components are separating curves and there is a subset $\partial_P X=\{\gamma_p\st p\in P\}$ of $|P|$ boundary components of $X$ such that $\gamma_p$ corresponds to $p$ for every $p$.

\subsection{(Dis)continuity of the action}
The first result we prove is that in general the action of the mapping class group on $\vA_2(\Sigma,P)$ (where $P$ is mapping class group invariant) is not continuous.

\begin{prop}\label{prop:discontinuous-action}
Let $\Sigma$ be a surface and $P$ a mapping class group invariant collection of ends of size at least three. Then the action of $\MCG(\Sigma)$ on $\vA_2(\Sigma,P)$ is continuous if and only if all ends in $P$ are punctures.
\end{prop}

\begin{proof}
Denote by $F$ the action $\MCG(\Sigma)\times \vA_2(\Sigma,P)\to \vA_2(\Sigma,P)$.

Suppose first that there is an end $e\in P$ which is not a puncture. Let $\alpha\in\vA_2(\Sigma,P)$ with $e\in\partial \alpha$ and consider $F^{-1}(\alpha)$. To show that the action is not continuous, it is enough to show that this is not an open set, i.e.\ that there is a point in $F^{-1}(\alpha)$ which doesn't have any open neighbourhood contained in $F^{-1}(\alpha)$. The point we will consider is $(\id,\alpha)$. If $F^{-1}(\alpha)$ contains an open neighbourhood of $(\id,\alpha)$, then there is a finite set $A$ of homotopy classes of essential simple closed curves such that $U_A\times\{\alpha\}\subset F^{-1}(\alpha)$. We will show that no such set is contained in the preimage of $\alpha$.

Let $A$ be any finite set of homotopy classes of essential simple closed curves. By the assumption on $e$, $\alpha$  is not contained in the subsurface filled by $A$. So there is an essential simple closed curve $b$ disjoint from all curves in $A$ and intersecting $\alpha$ nontrivially. Then the Dehn twist $T_b$ around $b$ is in $U_A$ but $(T_b,\alpha)\notin F^{-1}(\alpha)$.

Conversely, suppose all ends in $P$ are punctures and let $\alpha\in\vA_2(\Sigma,P)$. Then there is a finite collection of curves $A$ such that $\alpha$ is contained in the finite-type subsurface spanned by $A$. Let $(\phi,\beta)\in F^{-1}(\alpha)$. Then $\phi\cdot U_{\phi^-1(A)}\times \{\beta\}$ is an open neighborhood of $(\phi,\beta)$ contained in $F^{-1}(\alpha)$.
\end{proof}

We note that the cardinality of the set of vertices of our graph is also determined by whether $P$ contains only punctures or not:
\begin{lemma}\label{lem:uncountable}
The set of vertices of $\vA_2(\Sigma,P)$ is countable if all ends of $P$ are punctures. Otherwise it has the cardinality of the continuum.
\end{lemma}

\begin{proof}
Fix an exhaustion by properly embedded finite-type subsurfaces $X_1\supset X_2\supset \dots$ with essential boundary components.

If all ends of $P$ are punctures, any arc is contained in some $X_n$. Since every finite-type surface contains countably many arcs, the result follows.

If instead there is $e\in P$ which is not a puncture, let $\alpha$ be an arc in $\vA_2(\Sigma,P)$ with $e$ as an endpoint. The assumption on $e$ implies that there is a sequence of pairwise disjoint distinct simple closed curves $\gamma_n$, each intersecting $\alpha$. Then for any sequence of integers $k_n$,
$\left(\prod_{n\in\ZZ} T_{\gamma_n}^{k_n}\right)(\alpha)$ is an arc in $
\vA_2(\Sigma,P)$ and two distinct sequences give different arcs (where $T_\gamma$ is the Dehn twist about a curve $\gamma$). So $\vA_2(\Sigma,P)$ is uncountable. To show that its cardinality is not larger than the cardinality of the continuum, note the following:  given any arc, its homotopy class is determined by the intersections with the $X_n$ of a representative in minimal position with the $X_n$. In each $X_n$, the set of homotopy classes of multiarcs, where the homotopies fix the boundary pointwise, has the cardinality of the continuum. %So the set of vertices of $\vA_2(\Sigma,P)$ also has the cardinality of the continuum. 
\end{proof}

\subsection{Unicorns}

In this section we will adapt the unicorn construction from \cite{HPW_Slim} to the setting which allows for sets of arcs to have infinite intersection.

Let $\alpha,\beta$ be two 2-ended arcs which are not disjoint. To construct unicorns, we will always assume that we have chosen representatives of the arcs in minimal position.

A \emph{unicorn in $\alpha$ and $\beta$} is an arc of the form $a\cup b$, where $a$ is the closure of a connected component of $\alpha\sm p$, $b$ is a connected component of $\beta\sm p$ and $p$ is an intersection point of $\alpha$ and $\beta$. We call the point $p$ the \emph{corner} of the unicorn.

Note that for $a\cup b$ to be an arc, $a$ and $b$ can intersect only at $p$. In general, given two subarcs of $\alpha$ and $\beta$ with a common endpoint in $\Sigma$, there can be other intersection points, in which case the subarcs do not define a unicorn.

Denote by $\vU_2(\alpha,\beta)$ the collection of 2-ended unicorns in $\alpha$ and $\beta$ (i.e.\ the unicorns that are 2-ended arcs), together with $\alpha$ and $\beta$ themselves. If $\alpha$ and $\beta$ are disjoint, simply set $\vU_2(\alpha,\beta)=\{\alpha,\beta\}$. We will make use of the following crucial fact; which is the core of Remark 3.2 in \cite{HPW_Slim}, specialised to the case of 2-ended arcs. We include a proof for completeness.

\begin{lemma}\label{lem:crawlingunicorns}
Let $x=a\cup b\in\vU_2(\alpha,\beta)$. Then
\begin{itemize}
    \item if $\beta\sm b$ intersects $a$, there exists $y=a'\cup b'\in\vU_2(\alpha,\beta)$ which is adjacent to $x$, such that $a'\subsetneq a$ and $b'\supsetneq b$. In particular, the corner of $y$ belongs to the interior of $a$.
    \item if $\beta\sm b$ and $a$ are disjoint, then $x$ is disjoint from $\beta$.
\end{itemize}
\end{lemma}
\begin{proof}
It is clear that if $\beta\sm b$ and $a$ are disjoint, then so are $x$ and $\beta$. So let us assume that $\beta\sm b$ and $a$ are not disjoint. Let $p$ be the corner of $x$ and $q$ the first intersection with $a$ that occurs following $\beta\sm b$ from $p$. Then we can set $a'$ to be the subarc of $a$ from the endpoint at infinity to $q$ and $b'$ to be the union of $b$ with the segment of $\beta\sm b$ between $p$ and $q$. By construction, this defines a unicorn $y$ with the required properties.  
\end{proof}
We are interested in the following consequence:
\begin{cor}\label{cor:unicornpath}
If $x=a\cup b\in\vU_2(\alpha,\beta)$ and $|(\beta\sm b)\cap a|<\infty$, then there is a path in $\vU_2(\alpha,\beta)$ between $x$ and $\beta$.
\end{cor}

In particular this implies that if $\alpha$ and $\beta$ have finite intersection, then $\vU_2(\alpha,\beta)$ yields a connected subgraph. However, our arcs can intersect infinitely many times, in which case $\vU_2(\alpha,\beta)$ might not give a connected subgraph (see Figure \ref{fig:unicornpaths}).

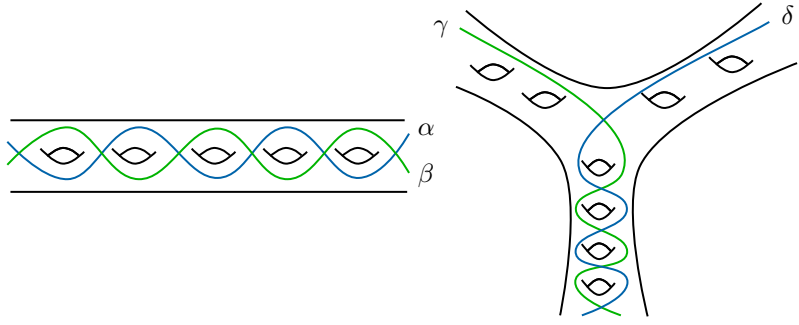
\begin{figure}
\centering
\begin{overpic}{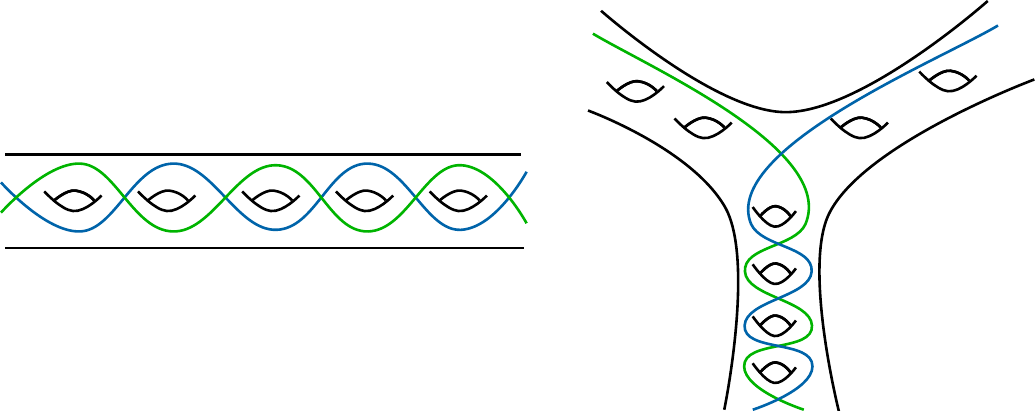}
\put(52,23){$\alpha$}
\put(52,17){$\beta$}
\put(54,36){$\gamma$}
\put(98,37){$\delta$}
\end{overpic}
\caption{The graph spanned by $\vU_2(\alpha,\beta)$ is disconnected, while that spanned by $\vU_2(\gamma,\delta)$ is connected}
\label{fig:unicornpaths}
\end{figure}

On the other hand, $\vU_2(\alpha,\beta)$ is never reduced to only $\alpha$ and $\beta$, provided that $\alpha$ and $\beta$ intersect.

\begin{lemma}\label{lem:unicornsexist}
Let $\alpha,\beta$ be non-disjoint 2-ended arcs.  Then $\vU_2(\alpha,\beta)$  contains  an arc $x\neq \alpha,\beta$.
\end{lemma}

\begin{proof}
Let $p$ be any intersection point of $\alpha$ and $\beta$ and let $a \subset \alpha$ and $b \subset \beta$ with endpoint $p$ such that the other two endpoints of $a,b$ are distinct ends $e_\alpha,e_\beta$.  If $a \cup b$ is an arc we are done. If it is not, then $a$ and $b$ intersect outside of $p$ and since they have different endpoints at infinity they intersect finitely many times. So we can set $q$ to be the first intersection that occurs when travelling from $e_\alpha$ to $p$ following $a$. Define $a'$ to be the subarc of $a$ with endpoints $e_\alpha$ and $q$ and $b'$ the subarc of $\beta$ (containing $b$) with endpoints $e_\beta$ and $q$. By construction $a'\cap b'=\{q\}$ and hence $a'\cup b'$ is a unicorn in $\vU_2(\alpha,\beta)$.
\end{proof}

Using this, we can show that even when $\vU_2(\alpha,\beta)$ does not give a connected subgraph, its 1-neighbourhood (the set of all vertices at distance at most one from some $x\in \vU_2(\alpha,\beta)$, denoted $N_1(\vU_2(\alpha,\beta))$) does.

\begin{lemma}
Suppose $|P|\geq 3$. For any $\alpha,\beta \in \vA_2(\Sigma,P)$, the full subgraph spanned by $N_1(\vU_2(\alpha,\beta))$ is connected.
\end{lemma}

\begin{proof}
If $\alpha$ and $\beta$ are disjoint, the result is clear. So we can assume that $\alpha$ and $\beta$ intersect and hence, by Lemma \ref{lem:unicornsexist}, $\vU_2(\alpha,\beta)\supsetneq \{\alpha,\beta\}$. It suffices then to show that given any unicorn $x_1 = a_1 \cup b_1\neq \alpha,\beta$ we can find a path from $x_1$ to $\beta$ that is contained in $N_1(\vU_2(\alpha,\beta))$. Since the definition of 2-ended unicorn is symmetric with respect to $\alpha$ and $\beta$, this will show that there is a path to $\alpha$ as well.

Let $p_1$ be the corner of $x_1$ and $b_0$ be the other half of $\beta$, that is, $\beta = b_0 \cup b_1$.  If $|a_1\cap b_0|<\infty$, Corollary \ref{cor:unicornpath} allows us to conclude. If $|a_1 \cap b_0| = \infty$, by Lemma \ref{lem:crawlingunicorns} we obtain an infinite sequence of unicorns $x_k=a_k\cup b_k$ with corners $p_k$ such that $a_k\subsetneq a_{k-1}$ and $x_k$ is disjoint from $x_{k-1}$ for every $k\geq 2$.  Moreover, since $|a_1 \cap b_0| = \infty$, we have for any $k$ that $\partial x_k=\partial \beta= \{e, e'\}$, where $e$ is the endpoint at infinity of $a_1$ and as the arcs are proper, $a_1 \cap b_0$ is discrete and hence $p_k\to e$. Let $S$ be a compact subsurface with boundary corresponding to $P$ and such that $|\partial S\cap\beta|=2$. As the corners $p_k$ of the unicorns converge to $e$, there is an index $k$ such that $p_k\in\Sigma\sm S$. If $X$ is the flare surface in the complement of $S$ containing $e$, $(\Sigma\sm X)\cap x_k=(\Sigma\sm X)\cap\beta$. Furthermore $S\cap\beta=S\cap x_k$ is a single arc and thus there exists another arc $c$ in $S$ disjoint from $\beta$, connecting boundary components corresponding to $e$ and $f$, where $f\in P\sm \{e, e'\}$. We can then extend $c$ to an arc $\gamma\in\vA_2(\Sigma,P)$ with $\gamma\cap S=c$ and disjoint from both $\beta$ and $x_k$. In particular $\gamma\in N_1(\vU_2(\alpha,\beta))$ and $x_1,\dots, x_k,\gamma,\beta$ is a path as required.
\end{proof}

A direct consequence of the lemma is the connectivity of the graph we are interested in:
\begin{cor}\label{cor:connected}
If $|P|\geq 3$ then $\vA_2(\Sigma,P)$ is connected.
\end{cor}

\subsection{Uniform hyperbolicity}\label{sec:hyperbolicity}
The goal of this section is to prove uniform hyperbolicity of $\vA_2(\Sigma,P)$, by applying the ``guessing geodesics lemma'' (\cite[Prop.\ 3.1]{Bowditch_Uniform},
\cite[Theorem 3.15]{MS_Geometry}).

\begin{lemma}[Guessing geodesics lemma]\label{lem:ggl}
Let $X$ be a graph. Suppose that for every pair of vertices $x$ and $y$ there is an associated connected subgraph $A(x,y)$ containing $x$ and $y$. If there is $M>0$ such that:
\begin{enumerate}
\item if $x$ and $y$ have distance at most one, then $A(x,y)$ has diameter at most $M$, and
\item for every $x,y,z$, $A(x,y)$ is contained in the $M$-neighbourhood of $A(x,z)\cup A(y,z)$,
\end{enumerate}
then $X$ is $\delta$-hyperbolic, where $\delta$ depends only on $M$.
\end{lemma}

We will show that the lemma applies to our situation, with 
$$A(\alpha,\beta)=N_1(\vU_2(\alpha,\beta)).$$

\begin{lemma}\label{lem:gglapplies}
Assume $|P|\geq 3$. Then the $A(\alpha,\beta)$ defined above yield connected subgraphs such that:
\begin{enumerate}
    \item if $\alpha$ and $\beta$ have distance at most one, then $A(x,y)$ has diameter at most $3$, and
    \item for every $\alpha,\beta,\gamma$, $A(\alpha,\beta)$ is contained in the $3$-neighbourhood of $A(\alpha,\gamma)\cup A(\beta,\gamma)$.
\end{enumerate}
In particular, there is a uniform constant $\delta$ (independent of $\Sigma$ or $P$) such that if $|P|\geq 3$ then $\vA_2(\Sigma,P)$ is $\delta$-hyperbolic.
\end{lemma}

\begin{proof}
Connectivity of the subgraphs is established in Corollary \ref{cor:connected}.
If $\alpha$ and $\beta$ are disjoint, $A(\alpha,\beta)=N_1(\alpha)\cup N_1(\beta)$ and hence has diameter at most three.
The proof of the fact that the subgraphs form thin triangles is essentially the same as the proof of Lemma 5.7 in \cite{DFV_Big}, but we include it for the sake of completeness. 

Fix now $\alpha,\beta$ and $\gamma$ and let $x\in A(\alpha,\beta)$. What we need to show is that there is $y\in A(\alpha,\gamma)\cup A(\beta,\gamma)$ with $d(x,y)\leq3$.
By definition, $x$ is at distance at most $1$ from a unicorn $x_1=a\cup b\in\vU_2(\alpha,\beta)$ with corner $p_1$. If $\gamma\cap x_1=\emptyset$, we can set $y=\gamma$, so assume that $x_1$ and $\gamma$ are not disjoint.
\begin{description}
\item[Case 1] $\partial\gamma\neq \partial x_1$. Let then $e$ be an end of $\gamma$ and not of $x_1$. Follow $\gamma$ from $e$ to the first intersection $p_2$ with $x_1$, which belongs to either $a$ or $b$ -- without loss of generality, assume $p_2\in a$. Set $a'$ to be the subarc of $a$ with same endpoint at infinity and ending at $p_2$ and let $c$ be the subarc of $\gamma$ from $e$ to $p_2$. By construction $y=a'\cup c$ is a unicorn in $\vU_2(\alpha,\gamma)\subset A(\alpha,\gamma)$ disjoint from  $x_1$.
\item[Case 2] $\partial \gamma=\partial x_1$. Let then $\delta\in\vA_2(\Sigma,P)$ disjoint from $\gamma$ and with $\partial\delta\neq\partial \gamma$. Let $e_\delta\in\partial\delta\sm\partial\gamma$. So $\partial\delta\neq \partial x_1$ and as in Case 1 we find $x_2\in\vU_2(\alpha,\delta)\cup \vU_2(\beta,\delta)$ disjoint from $x_1$ and such that $e_\delta\in\partial x_2$. Without loss of generality, assume $x_2\in\vU_2(\alpha,\delta)$. Note that $\partial x_2\neq\partial\gamma$, because $e_\delta\in\partial x_2\sm\partial \gamma$. We can thus apply Case 1 again to find $y\in\vU_2(\alpha,\gamma)\cup\vU_2(\delta,\gamma)\subset A(\alpha,\gamma)$ which is disjoint from $x_2$, which implies that $d(x,y)\leq 3$. \hfill \qedhere
\end{description}
\end{proof}

\subsection{Infinite diameter}\label{sec:unbounded}
When $P$ is mapping class group invariant, $\MCG(\Sigma)$ acts on $\vA_2(\Sigma,P)$. To show $\MCG(\Sigma)$ acts with unbounded orbits, we will find embedded infinite-diameter subgraphs on which certain mapping classes act with unbounded orbits. For this, we slighly extend the definition of our graph to the case of surfaces with boundary: if $Q$ is a collection of boundary components of a surface $X$, we denote by $\vA_2(X,Q)$ the subgraph of the arc graph of $X$ spanned by arcs joining two distinct boundary components of $Q$.

\begin{lemma}\label{lem:qi-emb} Suppose $|P|\geq 3$ and let $X$ be a finite-type properly embedded subsurface  of $\Sigma$ with boundary separating $P$. Then there is a quasi-isometric embedding of $\vA_2(X, \partial_P X)$ into $\vA_2(\Sigma,P)$.
\end{lemma}

\begin{proof}
Denote by $\gamma_p\in\partial_P X$ the boundary component of $X$ corresponding to $p\in P$. For each $p\in P$, fix an arc $\delta_p\subset \overline{\Sigma\sm X}$ joining $\gamma_p$ to $p$ and intersecting $\gamma_p$ in exactly one point. We define the map $E:\vA_2(X,\partial_P X)\to \vA_2(\Sigma,P)$ as follows: given an arc $\alpha\in\vA_2(X,\partial_P X)$, let $\partial\alpha=\{\gamma_{p_1},\gamma_{p_2}\}$. For $i=1,2$, choose a component $\alpha_{p_i}$ of $\gamma_{p_i}\setminus (\alpha\cup\delta_{p_i})$ joining $\gamma_{p_i}$ to $\delta_{p_i}$ and glue $\alpha_{p_i}$ and $\delta_{p_i}$ to $\alpha$. We obtain an arc $E(\alpha)$ joining $p_1$ to $p_2$. 
\begin{figure}[ht]
\begin{center}
\begin{overpic}{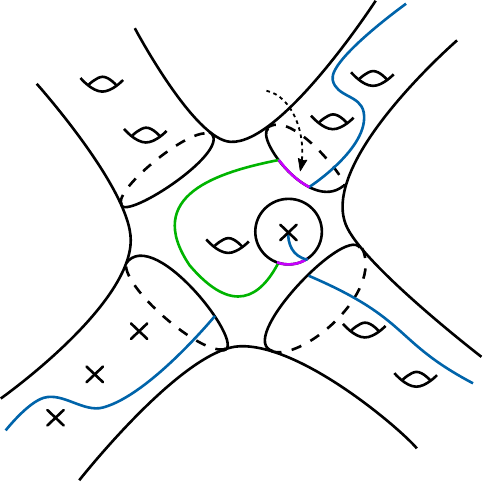}
\put(49,34){$\alpha$}
\put(46,81){$\alpha_p$}
\put(90,100){$p$}
\end{overpic}
\end{center}
\caption{Extending an arc in $\vA_2(X,\partial_P X)$ to an arc in $\vA_2(\Sigma,P)$}
\label{fig:extending}
\end{figure}

We claim that $E$ is a quasi-isometric embedding.  Denote by $d_X$ the distance $\vA_2(X,\partial_P X)$.  %If $\alpha,\beta\in\vA_2(X,\partial_P X)$ are disjoint, it is not hard to see that $d( E(\alpha) , E(\beta) )$ is at most $3$ in $\vA_2(\Sigma,P)$ (by possibly using other extensions, obtained using different subarcs of the boundary components).
Since the set of all possible extensions of an arc has diameter two and disjoint arcs of $\vA_2(X,\partial_P X)$ always admit disjoint extensions, if $\alpha,\beta\in\vA_2(X,\partial_P X)$ are disjoint, then $d( E(\alpha) , E(\beta) )\leq 3$ in $\vA_2(\Sigma,P)$. Hence for any $\alpha,\beta\in\vA_2(X,\partial_P X)$ we have
$$ d (E(\alpha), E(\beta))\leq 3 d_X(\alpha,\beta).$$

For the other inequality, let $E(\alpha)=\gamma_0,\gamma_1,\dots,\gamma_k=E(\beta)$ be a path in $\vA_2(\Sigma,P)$. Since $X$ has boundary separating $P$, for every $i$ the intersection $\gamma_i\cap X$ contains an arc $\gamma_i'$ joining two distinct boundary components in $\partial_P X$. So $\alpha=\gamma_0',\gamma_1',\dots,\gamma_k'=\beta$ induces a path in $\vA_2(X,\partial_P X)$.
This holds for every path in $\vA_2(\Sigma,P)$, so
$$d_X(\alpha,\beta) \leq d (E(\alpha),E(\beta)),$$
completing the proof.
\end{proof}

\begin{lemma}\label{lem:unbounded}
If $|P|\geq 3$ and $\Sigma$ is not the thrice-punctured sphere, then $\vA_2(\Sigma,P)$ has infinite diameter. If $P$ is $\MCG(\Sigma)$-invariant, then $\MCG(\Sigma)$ acts on $\vA_2(\Sigma,P)$ with unbounded orbits.
\end{lemma}

\begin{proof}
Choose a subsurface $X$ with boundary separating $P$: if $|P|\geq 4$, we can choose $X$ to be a $|P|$-holed sphere. If $|P|=3$, $\Sigma$ has positive genus or $\Ends(\Sigma)\geq 4$. In the first case, we can choose $X$ to be a a $3$-holed torus  and in the second case a $4$-holed sphere.

Pick an arc $\alpha_0\in\vA_2(X,\partial_P X)$; by applying a pseudo-Anosov of $X$ fixing the boundary components of $X$ pointwise we obtain a sequence of arcs $\alpha_k\in\vA_2(X,\partial_P X)$ whose distance from $\alpha_0$ goes to infinity in $\vA(X)$, thus in $\vA_2(X,\partial_P X)$. As a consequence, the diameter of $\vA_2(X,\partial_P X)$ is infinite and there are mapping classes of $X$ acting with unbounded orbits. By Lemma \ref{lem:qi-emb}, $\vA_2(X,\partial_P X)$ is quasi-isometrically embedded in $\vA_2(\Sigma,P)$, so the diameter of $\vA_2(\Sigma,P)$ is infinite as well.

If $P$ is invariant under the action of the mapping class group, $\MCG(\Sigma)$ acts on $\vA_2(\Sigma,P)$ and the compactly supported elements described above have unbounded orbits.
\end{proof}

\bibliographystyle{alpha}
\bibliography{references}
\end{document}